\documentclass[11pt,journal]{amsart}
\usepackage[top=2.5cm, bottom=2.5cm, left=3cm, right=3cm]{geometry}
\usepackage[T1]{fontenc}
\usepackage{lmodern}
\usepackage{amsmath,amssymb}
\usepackage{amsrefs}
\usepackage{microtype}
\usepackage{color}
\usepackage{needspace}
\usepackage{hyperref}
\hypersetup{colorlinks=true,linkcolor=blue,citecolor=blue,urlcolor=blue}
\newtheorem{theorem}{Theorem}[section]

\newtheorem{lemma}[theorem]{Lemma}
\newtheorem{proposition}[theorem]{Proposition}

\theoremstyle{remark}
\newtheorem{remark}[theorem]{Remark}
\DeclareMathOperator{\Tr}{Tr}
\DeclareMathOperator{\Rea}{Re}
\DeclareMathOperator{\supp}{supp}
\newcommand{\R}{\mathbb R}
\newcommand{\D}{\mathcal D}
\newcommand{\Sch}{\mathcal S}
\newcommand{\Ut}{U_\theta}
\newcommand{\taut}{\tau_\theta}

\newcommand{\Ft}{\mathcal F_\theta}
\newcommand{\LpQ}[1]{L_{#1}(\R_\theta^d)}
\newcommand{\norm}[1]{\lVert#1\rVert}
\newcommand{\abs}[1]{\lvert#1\rvert}
\allowdisplaybreaks[2]
\title
{Noncommutative sharp Hausdorff--Young inequalities}
\author{Hongsen Qiu}
\address{Institute for Advanced Study in Mathematics, Harbin Institute of Technology, Harbin, 150001, CHINA}
\email{chieughongsen@gmail.com}

\author{Xiao Xiong}
\address{Institute for Advanced Study in Mathematics, Harbin Institute of Technology, Harbin, 150001, CHINA}
\email{xxiong@hit.edu.cn}


\begin{document}
\begin{abstract}
We prove the sharp Hausdorff--Young inequality on quantum euclidean spaces. Our result implies the sharp Hausdorff--Young constants for the Weyl transform. The key ingredients are a sharp Hessian comparison inequality and a novel flow related to the mixed-norm of noncommutative Gabor transform; the Hessian comparison leads to the monotonicity of the flow. We also apply the same flow technique to prove the sharp Hausdorff--Young inequalities on Heisenberg groups and the affine group.
\end{abstract}

\maketitle

\section{Introduction}

\subsection{Introduction and main result}
Let $\mathcal F$ denote the usual Fourier transform on $\mathbb R^d$.
In 1975, Beckner \cite{Beckner1975} established the sharp
Hausdorff--Young inequality
\begin{align}\label{classical-sharp}
 \|\mathcal F(f)\|_{L_q(\mathbb R^d)}
 \le C_p^d\,\|f\|_{L_p(\mathbb R^d)},
 \qquad f\in L_p(\mathbb R^d),
\end{align}
where $1\le p\le2$, $1/p+1/q=1$, and
\[
 C_p=\sqrt{\frac{p^{1/p}}{q^{1/q}}}.
\]
The constant $C_p^d$, first known for $q=2,4,6,8\cdots$ by Babenko \cite{babenko}, is optimal, and equality is attained by suitable Gaussian functions.

For $f\in L_1(\mathbb C^n)$, define its Weyl transform by
\[
 \rho(f)
 =\frac{1}{(2\pi)^n}\int_{\mathbb R^n}\int_{\mathbb R^n}
 f(u+iv)e^{i(u\cdot P+v\cdot Q)}\,du\,dv,
\]
where $P=(P_1,\ldots,P_n)$ and $Q=(Q_1,\ldots,Q_n)$
are the operators defined by
\[
 P_j\varphi(x)=- i
               \frac{\partial}{\partial x_j}\varphi(x),
 \quad
 Q_j\varphi(x)=x_j\varphi(x),
 \qquad j=1,\ldots,n.
\]
In 1978, Klein and Russo
\cite{KleinRusso1978} proved the same constant for Weyl transform at $q=2,4,6,\cdots$, i.e., for even integers $q$, 
\[ \|\rho (f)\|_{\mathcal{S}_q} \leq C_p^{2n} \|f\|_{L_{p}(\mathbb{C}^n,\mu)}.\]
Here $\mathcal S_q$ denotes
the Schatten $q$-class in $\mathcal B(L_2(\mathbb R^n))$ with the norm
$\|T\|_{\mathcal S_q}=(\operatorname{Tr}|T|^q)^{1/q}$ and for $z=u+iv\in\mathbb C^n$ with the measure
$$d\mu(z)=(2\pi)^{-n}\,du\,dv.$$
For $q>2$ in this range, the optimal constant is not attained
by any nonzero function, but is saturated along a suitable
sequence of Gaussian functions.


Klein and Russo's work provides a natural starting point
for the noncommutative problem considered here. Later, a partial result to this problem was given in Lieb \cite{Lieb1990} in the form of ambiguity functions. In fact, Lieb's result is equivalent to 
$$ \|f\|_{L_{q}(\mathbb{C},\mu)} \leq \Big(\frac{2}{q}\Big)^{1/q} \|\rho(f)\|_{\mathcal{S}_p},\quad {\rm rank}\,\rho(f)=1$$
and the conjugate result
$$ \|\rho (f)\|_{\mathcal{S}_q} \leq \Big(\frac{2}{q}\Big)^{1/q} \|f\|_{L_{p}(\mathbb{C},\mu)},\quad {\rm rank}\,\rho(f)=1.$$
The most recent progress was made by Cowling, Martini, M\"uller and Parcet in \cite{CowlingEtAl2019}, where they proved that, for polyradial $f\in C_c^{\infty}(\mathbb{C}^n)$,
$$\norm{\rho(f)}_{\mathcal{S}_q}\le C_p^{2n}\norm{f(\xi )e^{\frac{|\xi|^2}{4}}}_{L_p(\mathbb{C}^n,\mu)}. $$
The unrestricted global sharp results for general $q$  remained
open there.

\medskip

In this paper, we are going to prove
\begin{equation}\label{eq:introHY}
 \norm{\rho(f)}_{\mathcal{S}_q}\le C_p^{2n}\norm{f}_{L_p(\mathbb{C}^n,\mu)}
\end{equation}
with the sharp constant for all $1\leq p\leq 2$. Equivalently, we work in the framework of quantum euclidean
spaces. Let $\theta$ be a $d$-dimensional skew-symmetric real matrix.
For $\xi\in\mathbb R^d$, define the unitary operator
\[
 (U_\theta(\xi)g)(x)
 =e^{\frac{i}{2}\langle\xi,\theta x\rangle}g(x-\xi),
 \qquad g\in L_2(\mathbb R^d).
\]
These operators satisfy the Weyl relations
\begin{align}\label{weyl-relation}
 U_\theta(\xi)U_\theta(\eta)
 =e^{\frac{i}{2}\langle\xi,\theta\eta\rangle}U_\theta(\xi+\eta),
 \qquad
 U_\theta(\xi)^*=U_\theta(-\xi).
\end{align}
Write $L_\infty(\mathbb R_\theta^d)$ for the von Neumann algebra generated by $\{U_\theta(\xi)\}_{\xi\in\mathbb R^d}$, and set
\[
 U_\theta(f)
 =\int_{\mathbb R^d}f(\xi)U_\theta(\xi)\,d\xi,
 \qquad f\in\mathcal S(\mathbb R^d).
\]
This algebra admits a canonical faithful normal  semifinite trace $\tau_\theta$
characterized by $\tau_\theta(U_\theta(f))=f(0)$.
The associated noncommutative $L_p$-spaces $L_p(\mathbb R_\theta^d)$ are equipped with the norms
\[
 \|x\|_{L_p(\mathbb R_\theta^d)}
 =\tau_\theta(|x|^p)^{1/p},
 \qquad 1\le p<\infty,
\]
and with the operator norm for $p=\infty$.
Our main result is
\begin{theorem}
Let $1\le p\le2, \frac1p+\frac1q=1$. For every \(f\in L_p(\R^d)\),
\begin{align}\label{main}
 \|U_\theta(f)\|_{L_q(\mathbb R_\theta^d)}
 \le C_p^d\,\|f\|_{L_p(\mathbb R^d)}.
\end{align}
The constant \(C_p^d\) is sharp.
\end{theorem}

To recover \eqref{eq:introHY}, observe that the Weyl unitaries
$D(u,v)=e^{i(u\cdot P+v\cdot Q)}$ satisfy
\[
 D(u,v)D(u',v')
 =e^{\frac{i}{2}(u\cdot v'-v\cdot u')}
  D(u+u',v+v'),
 \qquad
 D(u,v)^*=D(-u,-v).
\]
This coincides with the quantum euclidean space when $d=2n$ and
\[
 \theta=
 \begin{pmatrix}
  \mathbf{0}&I_n\\
  -I_n& \mathbf{0}
 \end{pmatrix}.
\]

\subsection{Proof strategy}

One main ingredient for the proof of \eqref{main} is the following monotone flow. For \(A> 0\), put
\begin{equation}\label{eq:localization}
 \mathcal L_A(f)=\left(\frac{pA}{\pi}\right)^{\frac{d}{2}}\int_{\R^d}
  \norm{\Ut\bigl(f(\cdot) e^{-A\abs{\,\cdot-s\,}^2}\bigr)}_{L_q(\R_\theta^d)}^{p}
 \,ds.
\end{equation}
We will show that \(\mathcal L_A(f)\) is nondecreasing in \(A\), and that the endpoint limits
\[
 \lim_{A\to 0^+}\mathcal L_A(f)
 =\|U_\theta(f)\|_{L_q(\mathbb R_\theta^d)}^p,
 \qquad
 \lim_{A\to\infty}\mathcal L_A(f)
 =C_p^{dp}\|f\|_{L_p(\mathbb R^d)}^p.
\]
This gives \eqref{main} directly. When $\theta=0$, the map $U_\theta$ identifies with the
classical Fourier transform, and $\Ut\bigl(f(\cdot) e^{-A\abs{\,\cdot-s\,}^2}\bigr)$ is indeed the Gabor transform. Hence, for general $\theta$, we can view $\Ut\bigl(f(\cdot) e^{-A\abs{\,\cdot-s\,}^2}\bigr)$ as the noncommutative Gabor transform.

\medskip

Beckner's proof \cite{Beckner1975}  for \eqref{classical-sharp} is based on the complex
hypercontractivity of the Ornstein--Uhlenbeck semigroup. It is therefore natural to seek an analogous argument using
the quantum Ornstein--Uhlenbeck semigroup. However, Beigi and Rahimi-Keshari's work \cite{beigimeta} reveals that the hypercontractivity is different from the classical case---it contains a temperature constant and can only recover the classical result when the temperature tends to $\infty$. Due to the temperature constant, we do not see how to carry out Beckner's
argument in the present setting to prove the noncommutative sharp Hausdorff-Young inequality \eqref{eq:introHY}.

Our approach is motivated by flow proofs of sharp functional
inequalities, including the semigroup proof of the Gaussian isoperimetric
inequality by Bakry and Ledoux \cite{BakryLedoux1996} and the heat-flow proof of the sharp
Young's inequality by Bennett and Bez \cite{BennettBez2009}.
The underlying principle is to construct a monotone quantity whose
endpoint values give the two sides of the desired inequality.
For the Hausdorff--Young inequality, heat flow is a natural
candidate. However, Bennett, Bez and Carbery \cite{Bennet2008}
showed that the natural heat-flow monotonicity fails even in the
classical setting when $q>2$ is not an even integer.
Ivanisvili and Volberg \cite{Paata} observed that the sharp
Hausdorff--Young inequality follows from the monotonicity
result in \cite{Janson}.
They referred to the corresponding flow as the
Beckner--Janson flow \cite[Eq.(15)]{Paata}. When $\theta=0$, our flow \eqref{eq:localization} is more closely aligned with the Beckner--Janson flow.

More precisely, for $\theta=0$, our flow admits a time--frequency interpretation
in terms of mixed norms of the Gabor transform. Let
$g_A(t)=e^{-A|t|^2}$ and define the Gabor transform
\[
V_{g_A}f(s,\omega)
=\int_{\mathbb R^d}f(t)g_A(t-s)e^{- i t\cdot\omega}\,dt.
\]
With this Fourier convention, the commutative specialization of our flow becomes
\[
\mathcal L_A(f)^{1/p}
=\frac{\|V_{g_A}f\|_{L_p(ds;L_q((2\pi)^{-d}d\omega))}}{\|g_A\|_p}
=\left(\frac{pA}{\pi}\right)^{\frac{d}{2p}}
\left[
\int_{\mathbb R^d}
\left(\int_{\mathbb R^d}|V_{g_A}f(s,\omega)|^q\,\frac{d\omega}{(2\pi)^d}\right)^{p/q}
ds
\right]^{1/p}.
\]
Thus the flow is a normalized mixed norm, with the frequency
$L_q$ norm taken before the time-domain $L_p$ norm.
As $A\to 0^+$, the Gaussian window spreads in time and
concentrates in frequency; as $A\to\infty$, it conversely concentrates in
time and spreads in frequency. The endpoint limits make this transition precise:
\[
\lim_{A\to 0^+}\mathcal L_A(f)^{1/p}
=\|\mathcal F(f)\|_q,
\qquad
\lim_{A\to\infty}\mathcal L_A(f)^{1/p}
=C_p^d\|f\|_p.
\]
Consequently, the monotonicity of such flow connects the frequency
norm of $f$ to its time-domain norm and directly yields Beckner's
sharp Hausdorff--Young inequality. 

\medskip

Another main ingredient is the following sharp Hessian comparison:
\begin{equation}\label{intro-hessian}
D^2\Phi(Y)[B,B]
\ge (p-1)D^2\Phi(Y)[iB,iB],
\qquad
\Phi(Y)=\frac{1}{q}\tau(|Y|^q),
\end{equation}
where \(1<p\le 2\), \(q=p/(p-1)\), and
\(Y,B\in L_q(\mathcal M,\tau)\) for a semifinite tracial
von Neumann algebra \((\mathcal M,\tau)\). This estimate plays the role of the two-point hypercontractivity in Beckner's proof. Rather than relying on pointwise
scalar estimates, it provides a sharp geometric comparison of the
Hessian of \(\Phi\) in directions related by the complex structure
\(B\mapsto iB\), and thus captures the complex geometry of noncommutative
\(L_q\)-spaces.

We will not discuss possible applications of the above Hessian comparison \eqref{intro-hessian} to hypercontractivities of (quantum) Markov semigroups in this paper, for it needs further discussion on the principle of maximum entropy under the energy constraints induced by the semigroup generators, which goes beyond the scope of this paper. But thanks to the generality of this Hessian comparison, we will be able to extend the flow trick to Heisenberg groups and affine group and prove the monotonicity of the flow for free, to obtain the sharp Hausdorff--Young inequalities for them.

\subsection{Extension of the flow trick to noncommutative groups}
We also apply our flow trick to prove the sharp Hausdorff--Young inequalities for Heisenberg groups and affine group, both were studied in \cite{KleinRusso1978} for even $q$.
For Heisenberg groups, according to \cite{CowlingEtAl2019}, our result \eqref{main} implies the sharp Hausdorff-Young inequality directly. Neverthless, we will adapt the flow \eqref{eq:localization} to Heisenberg groups, and give a proof without invoking the sharp Hausdorff-Young inequality \eqref{main}, nor the relation between sharp Hausdorff-Young inequality for Weyl transform and that on Heisenberg groups given in \cite{CowlingEtAl2019}.
For the affine group $\mathbb R\rtimes\mathbb R_+$, due to its nonunimodularity, more effort is needed for adapting the flow and calculating the two endpoint limits of the flow. 

\medskip

Identify the Heisenberg group $\mathbb H_n$ with
$\mathbb C^n\times\mathbb R$, equipped with the group law
\[
 (z,t)(z',t')
 =\left(z+z',t+t'+\frac12\operatorname{Im}(\overline z\cdot z')\right)
\]
and Haar measure $dz\,dt$. For $\lambda\in\mathbb R\setminus\{0\}$,
the Schr\"odinger representation $\pi_\lambda$ acts on
$L_2(\mathbb R^n)$ by
\[
 \bigl(\pi_\lambda(u+iv,t)\varphi\bigr)(x)
 =e^{2\pi i\lambda t
    +2\pi i\operatorname{sgn}(\lambda)\sqrt{|\lambda|}\,v\cdot x
    +\pi i\lambda u\cdot v}
 \varphi\bigl(x+\sqrt{|\lambda|}\,u\bigr),\qquad z=u+iv.
\]
For $f\in L_1(\mathbb H_n)$, define
\[
 \mathcal{F}_{\mathbb H_n}(f)(\lambda)=\pi_\lambda(f)
 =\int_{\mathbb C^n}\int_{\mathbb R}
 f(z,t)\pi_\lambda(z,t)\,dt\,dz.
\]
The group von Neumann algebra $\mathcal M_{\mathbb H_n}$ is identified
with the direct integral of $\mathcal B(L_2(\mathbb R^n))$ over
$\mathbb R\setminus\{0\}$, with the Plancherel trace
\[
 \tau_{\mathbb H_n}(X)
 =\int_{\mathbb R\setminus\{0\}}
 \Tr(X(\lambda))|\lambda|^n\,d\lambda,
 \qquad X\ge0
\]
and the $L_q$ norm 
\[
 \|X\|_q
 =\left(\int_{\mathbb R\setminus\{0\}}
 \|X(\lambda)\|_{\mathcal S_q}^q|\lambda|^n\,d\lambda\right)^{1/q}.
\]
For $A>0$ and $f\in L_p(\mathbb H_n)$, define
\begin{equation*}
 \mathcal L_A^{\mathbb H_n}(f)
 =\left(\frac{pA}{\pi}\right)^n\int_{\mathbb C^n}
 \left\|\mathcal{F}_{\mathbb H_n}\big(f(z,t)e^{-A|z-s|^2}\big)\right\|_q^p\,ds.
\end{equation*}
We will deduce the monotonicity of $ \mathcal L_A^{\mathbb H_n}(f)$ from \eqref{intro-hessian},
and calculate the two endpoint limits
\begin{equation*}
 \lim_{A\to0^+}\mathcal L_A^{\mathbb H_n}(f)=\|\mathcal{F}_{\mathbb H_n}(f)\|_q^p,
 \qquad
 \lim_{A\to\infty}\mathcal L_A^{\mathbb H_n}(f)=C_p^{2np}\int_{\mathbb C^n}
 \left(\int_{\mathbb R}|\widehat f_t(z,\xi)|^q\,d\mu(\xi)\right)^{p/q}
 dz,
\end{equation*}
where $\widehat f_t(z,\xi)
=\int_{\mathbb R}f(z,t)e^{-it\xi}\,dt,$ is equipped with $d\mu(\xi)=(2\pi)^{-1}d\xi$. Applying the classical one dimensional Hausdorff--Young inequality to $\widehat f_t(z,\xi)$ will finally give the sharp inequality
\[
 \|\mathcal{F}_{\mathbb H_n}(f)\|_q\le C_p^{2n+1}\|f\|_{L_p(\mathbb H_n)},
 \qquad 1\le p\le2.
\]

We will see in Section~\ref{sec-heisenberg} that most steps of the proof are similar to those for quantum euclidean spaces, but there is one notable difference: the endpoint $A=\infty$ does not give the norm $\|f\|_{L_p(\mathbb H_n)}$, but give an intermediate mixed norm; one further Fourier transform is required. This phenomenon will also happen to the affine group which we introduce below.

\medskip

For the affine group $G=\mathbb R\rtimes\mathbb R_+$ with group law
\((b,a)(b',a')=(b+ab',aa')\), we use right
Haar measure $d\mu_R(b,a)=db\,da/a$.
Its modular function is $\Delta(b,a)=a^{-1}$.
The two representations of $G$ acting on
$\mathcal H=L_2(\R_+,dt/t)$ are given by
\[
 (\pi_\pm(b,a)\xi)(t)=e^{\mp2\pi ibt}\xi(at),
 \qquad \xi\in\mathcal H.
\]
Let $M_t$ be the positive operator $(M_t\xi)(t)=t\xi(t)$. For
$f\in L_1(G,\mu_R)$, put
\[
 \pi_{\pm,R}(f)=\int_G f(b,a)\pi_\pm(b,a)\,d\mu_R(b,a).
\]
For $1<p\le2$ and $q=p'$, define initially for $f\in C_c^\infty(G)$
\begin{equation*}
 \mathcal J_q(f)
 =\left(\overline{M_t^{1/q}\pi_{+,R}(f)},
         \overline{M_t^{1/q}\pi_{-,R}(f)}\right).
\end{equation*}
where $\overline{\hspace{1mm}\cdot\hspace{1mm}}$ denotes operator closure. The target von Neumann algebra
and its trace are
\[
 \mathcal M_G=\mathcal B(\mathcal H)\oplus\mathcal B(\mathcal H),
 \qquad \tau_G=\Tr\oplus\Tr.
\]
Thus
\[
 \|\mathcal J_q(f)\|_q^q
 =\|\overline{M_t^{1/q}\pi_{+,R}(f)}\|_{\mathcal S_q}^q+\|\overline{M_t^{1/q}\pi_{-,R}(f)}\|_{\mathcal S_q}^q.
\]
This $\mathcal J_q$ is the weighted Fourier transform
 defined in Section~\ref{sec-affine}.

For $f\in L_p(G,\mu_R)$
Write
\[
 F(b,u)=f(b,e^u),\qquad \mathcal T_q(F)=\mathcal J_q(f),
 \qquad
 \widehat F_b(\xi,u)=\int_\R F(b,u)e^{- ib\xi}\,db.
\]
Then $\|f\|_{L_p(G,\mu_R)}=\|F\|_{L_p(\R^2)}$. Using this $F$ and $\mathcal T_q(F)$, we define the flow
\begin{equation*}
 \mathcal L_A^G(F)
 =\left(\frac{pA}{\pi}\right)^{1/2}
 \int_\R
 \|\mathcal T_q(F(b,u)e^{-A|u-s|^2})\|_q^p\,ds.
\end{equation*}
The monotonicity of $\mathcal L_A^G(F)$ is again guaranteed by \eqref{intro-hessian}, and the two endpoint limits are 
\begin{equation*}
 \lim_{A\to0^+}\mathcal L_A^G(F)=\|\mathcal T_q(F)\|_q^p,
 \qquad
 \lim_{A\to\infty}\mathcal L_A^G(F)
 =C_p^p\int_\R\|\widehat F_b(\cdot,s)\|_{L_q(\R,d\mu)}^p\,ds.
\end{equation*}
Here $\widehat F_b(\xi,s)$ is equipped with $d\mu(\xi)=(2\pi)^{-1}d\xi$. Applying the classical one dimensional sharp Hausdorff--Young inequality to $\widehat F_b$ gives 
 \[
 \|\mathcal T_q(F)\|_q\le C_p^2\|F\|_{L_p(\mathbb{R}^2)},
 \qquad 1\le p\le2,
\]
whence the sharp Hausdorff--Young inequality
\[
 \|\mathcal J_q(f)\|_q\le C_p^2\|f\|_{L_p(G,\mu_R)}.
\]

\medskip

Given the successful applications of our method to the two groups above, we believe that the flow trick together with the Hessian comparison can be extended to more locally compact groups. However, given the amount of computation involved and the limits of our expertise, we do not consider other groups in this paper.

\medskip

\noindent\emph{\bf Organization.}
The paper is organized as follows.
Section~\ref{sec-pre} recalls the framework of quantum euclidean space and
the Gaussian norm formulas used in the proof.
Section~\ref{complex-dom} establishes complex dominance of the second-order Fr\'echet derivative on noncommutative $L_q$-spaces associated with a semifinite von Neumann algebra.
In Section~\ref{sec-main}, we prove the monotonicity and endpoint limits
of \eqref{eq:localization}, and deduce the sharp
Hausdorff--Young inequality.
Sections~\ref{sec-heisenberg} and~\ref{sec-affine} apply the same
argument to Heisenberg groups and the affine group, respectively.
Appendix~\ref{app} records the orthogonal basis of quantum euclidean spaces. 

Throughout the paper, unless otherwise stated, let \(1\leq p \leq 2\) and write
\[
 q=\frac{p}{p-1},\qquad C_p=\sqrt{p^{1/p}q^{-1/q}}
\]
with the convention $C_p=1,q=\infty$ when $p=1$.

\section{Quantum euclidean space}\label{sec-pre}
\subsection{Calculus on quantum euclidean space}
The quantum euclidean space framework used below follows that of \cite{MSX2020}, see also \cite[Section 6]{Fedor2020}. The relation with the Weyl transform and the Hausdorff--Young problem
is discussed in \cite{KleinRusso1978,CowlingEtAl2019}; standard facts
about the Weyl representation can be found in \cite{Folland1989}. We briefly review the definition and basic properties here.

Let $\theta$ be a real skew-symmetric $d\times d$ matrix.
Define the following family of unitary operators on $L_2(\mathbb R^d)$
\[
 (\Ut(\xi)u)(x)
 =e^{\frac i2\langle\xi,\theta x \rangle}u(x-\xi),\qquad u\in L_2(\mathbb R^d).
\]
It is easily verified that the family $\{\Ut(\xi)\}_{\xi \in\R^d}$ satisfies
\begin{equation}\label{eq:weyl}
 \Ut(\xi)\Ut(\eta)
 =e^{\frac i2 \langle\xi,\theta\eta \rangle}
 \Ut(\xi+\eta),\qquad
 \Ut(\xi)^*=\Ut(-\xi).
\end{equation}
The von Neumann subalgebra of $\mathcal B(L_2(\mathbb R^d))$ generated by
$\{\Ut(\xi)\}_{\xi\in\mathbb R^d}$ is called a
quantum euclidean space, denoted by
$\R^d_\theta=L_\infty(\R^d_\theta)$.

With $\theta$ given as above, the algebra $L_\infty(\R^d_\theta)$ is
$*$-isomorphic to the algebra 
\begin{align}\label{iso}
    L_\infty(\mathbb R^d_\theta)\cong L_\infty(\mathbb R^{{\rm dim }({\rm ker}(\theta))})\overline{\otimes}\mathcal B(L_2(\mathbb R^{\frac{{\rm rank}(\theta)}{2}})).
\end{align}
If ${\rm dim }({\rm ker}(\theta))=0$, we call the quantum euclidean space  $L_\infty(\R^d_\theta)$ is non-degenerate.  In non-degenerate cases, $d$ must be even.

Let $f\in L_1(\mathbb R^d)$. We define
$\Ut(f)\in L_\infty(\mathbb R^d_\theta)$ as the operator given by the
integral
\begin{equation}\label{eq:synthesis}
 \Ut (f)=\int_{\R^d} f(\xi)\Ut(\xi)\,d\xi.
\end{equation}
This integral is understood in the strong operator sense.
\begin{remark}The map $\Ut$ has other names and notations
in the literature. It is called the Weyl quantization map in
\cite{qtmath}. In \cite{qmath}, $\Ut$ is also essentially the same as the so-called Weyl transform there. See also \cite{ParcetMAMS}.
\end{remark}

There is a canonical semifinite normal trace $\tau_\theta$ on
$L_\infty(\mathbb R^d_\theta)$. Define
$$ \mathcal S(\mathbb R^d_\theta)=\{\Ut(f)\;
:\; f\in\mathcal S(\mathbb R^d) \}.$$ 
For $x=\Ut (f)\in\mathcal S(\mathbb R^d_\theta)$, we define
$\tau_\theta(x)$ as
\[
    \tau_\theta(x)=f(0).
\]
After matching the Weyl realizations with
the cocycle \eqref{eq:weyl}, our trace is \((2\pi)^{-d}\) times the
trace in \cite[Definition 2.6 and Lemma 2.7]{MSX2020}. 


The functional
$\tau_\theta:\mathcal S(\mathbb R^d_\theta)\to\mathbb C$ admits an extension
to a semifinite normal trace on $L_\infty(\mathbb R^d_\theta)$. For \(1\le p<\infty\), the noncommutative $L_p$-norm is
\[
 \norm{x}_{L_p(\mathbb R^d_\theta)}=\taut(\abs x^p)^{1/p}.
\]
When \(p=\infty\) it is the operator norm. We abbreviate this norm to \(\norm{x}_p\), and abuse the notation
\(\norm f_p\) for the usual $L_p$-norm when $f$ is a function.

For \(f\in\Sch(\R^d)\), we compute
$$ \Ut(f)^*=\int_{\R^d}\overline{f(\xi)}\Ut(\xi)^*\,d\xi=\int_{\R^d}\overline{f(-\xi)}\Ut(\xi)\,d\xi.$$
Hence we denote $f^{*}(\xi)=\overline{f(-\xi)}$. Define the twisted convolution by
$$ f*_\theta g(\xi)=\int_{\R^d}f(\xi-\eta)g(\eta) e^{\frac{i}{2}\langle \xi,\theta\eta\rangle}\,d \eta.$$
Then we have 
$$\Ut(f)\Ut(g)=\Ut(f*_\theta g).$$
It follows that
\begin{equation}\label{eq:plancherelpair}
 \taut\bigl(\Ut (g)^*\Ut (f)\bigr)=g^* *_{\theta}f(0)
 =\int_{\R^d}\overline{g(\xi)}f(\xi)\,d\xi.
\end{equation}
The resulting unitary extension from \(L_2(\R^d)\) onto \(L_2(\R_\theta^d)\)
is the Plancherel identity. The integral bound and Plancherel give
\[
 \norm{\Ut (f)}_\infty\le\norm f_1,
 \qquad \norm{\Ut (f)}_2=\norm f_2.
\]
Interpolation yields the non-sharp Hausdorff-Young inequality
\begin{equation}\label{eq:nonsharp}
 \norm{\Ut (f)}_q\le\norm f_p.
\end{equation}
In particular, \(\Ut:L_p(\R^d)\to L_q(\R_\theta^d)\) is defined by a unique
bounded extension.

Finally, the (reverse--)Fourier transform is initially defined on \(L_1(\R_\theta^d)\) by
\[
 \Ft (x)(\xi)=\taut\bigl(x\Ut(\xi)^*\bigr).
\]
The endpoint estimates and interpolation similarly yield
\[
 \|\mathcal F_\theta(x)\|_{L_q(\mathbb R^d)}
 \le \|x\|_{L_p(\mathbb R_\theta^d)},
 \qquad 1\le p\le2,\quad \frac1p+\frac1q=1.
\]
Thus $\mathcal F_\theta$ extends uniquely to a bounded map
from $L_p(\mathbb R_\theta^d)$ to $L_q(\mathbb R^d)$.

We need to introduce two isometries. The trace preserving automorphisms
\begin{equation}\label{eq:phase}
 T_t(\Ut(\xi))
 =e^{it\cdot\xi}\Ut(\xi),\qquad \xi,t\in\R^d,
\end{equation}
act isometrically on each \(L_p(\R_\theta^d)\) and satisfy
\[
 \Ut(e^{it\cdot\xi}f(\xi))=T_t(\Ut (f)).
\]
This is the so-called translation and it is used in the differential calculus of \cite[Section 3.1]{MSX2020}, characterizing the displacement.
Moreover, \eqref{eq:weyl} gives
\begin{equation}\label{eq:translation}
 \Ut(f(\cdot-s))
 =\Ut(s/2)(\Ut( f))\Ut(s/2).
\end{equation}
Since $\Ut(s/2)$ is unitary, for fixed $f$, the map
$$ s \longmapsto  \Ut(f(\,\cdot -s\,))$$
is $L_p$-norm-preserving. The identity extends from integrable
functions to \(L_p(\R^d)\) by \eqref{eq:nonsharp}.

We also need to briefly introduce the dilation. For more details, one can refer to \cite[Section 3]{MSX2020}. Define the dilation map
\begin{align*}
 \Psi_{\lambda}
    (\Ut(\xi)) =  U_{\lambda^2\theta}(\xi/\lambda).
\end{align*}
We then have 
\begin{align*}
 \Psi_{\lambda}: L_\infty(\R_\theta^d)&\longrightarrow L_\infty(\R_{\lambda^2\theta}^d),\\
    \Ut(f) &\longmapsto  \lambda^d U_{\lambda^2\theta}(f(\lambda\,\cdot\,))
\end{align*}
and the following identity:
\begin{align}\label{dia-id}
 \|\Psi_{\lambda} \Ut(f)\|_p=\|\lambda^d U_{\lambda^2 \theta}(f(\lambda\,\cdot \,))\|_p=\lambda^{\frac{d}{p}}\|\Ut(f)\|_p.
\end{align}

\subsection{A limit formulation for $\Ut$}
In this part, we are going to compute the limit $$\lim_{A\to \infty}\frac{\norm{\Ut( g_A)}_q}{\norm{g_A}_p}$$ 
for $g_A(\xi)=e^{-A\abs\xi^2}$. To this end, we need to analyze the $L_p$-norm on $L_\infty(\R_\theta^d)$ for different $\theta$. 

We begin with $\theta$ with standard form. When $\theta=\mathbf{0}_d$ is a $d$-dimensional zero matrix, $\Ut$ is exactly the usual Fourier transform; when $$\theta=\begin{pmatrix}0&-1\\1&0\end{pmatrix}=:\mathfrak{s},$$
the trace $\taut$ of $\LpQ{\infty}$ is exactly
\begin{equation*}
 \taut= \frac{1}{(2\pi)^{d/2}}\Tr.
\end{equation*}
For $\theta$ in the standard form
\begin{align}\label{standard-form}
\begin{pmatrix}
0 & -\lambda_1 &        &        &        &        &     &\\
\lambda_1 & 0 &        &        &        &        &     &\\
   &   & \ddots &        &        &        &     &\\
   &   &        & 0      & -\lambda_s      &        &     &\\
   &   &        & \lambda_s     & 0      &        &     &\\
   &   &        &        &        & 0     &     &\\
      &   &        &        &        &   & \ddots   & \\
          &   &        &        &        &   &   &0
\end{pmatrix}=\theta^\prime \oplus \mathbf{0}_t,\qquad \lambda_1,\cdots,\lambda_s>0
\end{align}
where $\pm i\lambda_1,\cdots,\pm i\lambda_s$ are nonzero eigenvalues of $\theta$ and $\mathbf{0}_t$ is the $t$-dimensional zero matrix; see the discussion in \cite[Section 6]{Fedor2020}.
 In this case, the trace $\taut$ of $\LpQ{\infty}$ is exactly
\begin{equation}\label{eq:traceTr}
 \taut=\frac{1}{(2\pi)^t}\left(\int_{\R^{{\rm dim }({\rm ker}(\theta))}}\,d\mu\right)\otimes \frac{\sqrt{|{\rm det}(\theta^\prime)|}}{(2\pi)^{s}}\Tr.
\end{equation}
Here $\Tr$ denotes the canonical trace of $\mathcal{B}(\mathcal H)$, where $\mathcal H=L_2(\mathbb R^{\frac{{\rm rank}(\theta)}{2}})$ is a separable Hilbert space. 

For general skew-symmetric real matrix $\theta$, there exists an orthogonal matrix $Q$ such that $\theta = Q\theta_0Q^{-1}$ where $\theta_0$ is in the form of \eqref{standard-form}. The corresponding change of variables induces a
trace-preserving $*$-isomorphism between $\mathbb R_\theta^d$ and
$\mathbb R_{\theta_0}^d$, and hence isometric isomorphisms between their
$L_p$ spaces for $1\le p\le\infty$. Thus, for the computation of $L_p$
norms, it suffices to consider $\theta$ in the form of \eqref{standard-form}.

\begin{proposition}\label{lem:gaussian}
Let \(g_A(\xi)=e^{-A\abs\xi^2}\) where $\xi\in\R^d$. The following limit holds:
\[
 \lim_{A\to \infty}\left(\frac{pA}{\pi}\right)^{\frac{d}{2p}}\norm{\Ut( g_A)}_q
 =\lim_{A\to \infty}\frac{\norm{\Ut( g_A)}_q}{\norm{g_A}_p}
 = C_p^d.
\]
\end{proposition}

\begin{proof}
We first prove the result when $d=2$ and 
\(
 \theta=\mathfrak{s}
\). In this case,  $\norm{U_{\mathfrak{s}}( g_A)}_q$ can be directly computed with the orthogonal basis introduced in Appendix~\ref{app}. By the properties of Laguerre functions, for $A>\frac{1}{4}$, we can write
$$ e^{-A|\xi|^2}=\frac{1}{2A+\frac{1}{2}}\sum_{n\geq 0} \left(\frac{2A-\frac{1}{2}}{2A+\frac{1}{2}}\right)^n L_n\left(\frac{|\xi|^2}{2}\right)e^{-\frac{|\xi|^2}{4}}=\frac{2\pi}{2A+\frac{1}{2}}\sum_{n\geq 0}\left(\frac{2A-\frac{1}{2}}{2A+\frac{1}{2}}\right)^n f_{nn}(\xi). $$
Hence 
$$ \norm{U_{\mathfrak{s}}( g_A)}_q^q=\frac{1}{2\pi} \Tr |U_{\mathfrak{s}}(e^{-A|\xi|^2})|^q=(2\pi)^{q-1}\left(2A+\frac{1}{2}\right)^{-q}\sum_{n\geq 0}  \left(\frac{2A-\frac{1}{2}}{2A+\frac{1}{2}}\right)^{qn}$$
Computing the geometric series gives
\[
\norm{U_{\mathfrak{s}}(e^{-A|\xi|^2})}_q=(2\pi)^{1-\frac{1}{q}}
\left[
\left(2A+\tfrac12\right)^q
-
\left(2A-\tfrac12\right)^q
\right]^{-\frac{1}{q}}.\]

The result can extend to general case easily. By the preceding discussion, we only need to consider
\[
 \theta=\bigoplus_{i=1}^{s}\lambda_i\begin{pmatrix}0&-1\\1&0\end{pmatrix}\oplus \mathbf{0}_t=\bigoplus_{i=1}^{s}\lambda_i\mathfrak{s}\oplus \mathbf{0}_t,\qquad  d=2s+t
\]
where $\pm i\lambda_1,\cdots,\pm i\lambda_s$ are the nonzero eigenvalues of $\theta$. For $\xi\in \R^d$, 
$$ e^{-A|\xi|^2}=e^{-A(\xi_1^2+\xi_2^2+\cdots+\xi_d^2)}.$$
We have 
\begin{align*}\Ut (e^{-A|\xi|^2})=&U_{\lambda_1\mathfrak{s}} (e^{-A(\xi_1^2+\xi_2^2)})\otimes\cdots\otimes U_{\lambda_s\mathfrak{s}}(e^{-A(\xi_{2s-1}^2+\xi_{2s}^2)})\\
 & \otimes U_0(e^{-A\xi_{2s+1}^2})\otimes\cdots\otimes U_0(e^{-A\xi_{d}^2})
\end{align*}
By the dilation property \eqref{dia-id}, we obtain
$$\norm{U_{\lambda_i \mathfrak{s}} (e^{-A|\xi|^2})}_q=\lambda_i^{-\frac{1}{q^\prime}}\norm{U_{\mathfrak{s}} (e^{-\frac{A}{\lambda_i}|\xi|^2})}_q=\lambda_i^{\frac{1}{q}}(2\pi)^{1-\frac{1}{q}}
\left[
\left(2A+\tfrac{\lambda_i}{2}\right)^q
-
\left(2A-\tfrac{\lambda_i}{2}\right)^q
\right]^{-\frac{1}{q}}$$
When $A\to \infty$, for arbitrary $\lambda_i>0$, we have
$$ \lambda_i^{\frac{1}{q}}(2\pi)^{1-\frac{1}{q}}
\left[
\left(2A+\tfrac{\lambda_i}{2}\right)^q
-
\left(2A-\tfrac{\lambda_i}{2}\right)^q
\right]^{-\frac{1}{q}}\sim q^{-\frac{1}{q}}\left(\frac{\pi}{A}\right)^{\frac{1}{p}}.$$
By the properties of the Gaussian functions under the usual Fourier transform, we obtain
$$ U_{0} (e^{-Ax^2})=\left(\frac{\pi}{A}\right)^{1/2}
  e^{-\frac{x^2}{4A}}.$$
In this case, the canonical trace is
\(
 \tau_0=\frac{1}{2\pi}
           \int_{\mathbb R}\,dx
\). For $1\le q<\infty$,
\begin{align*}
 \|U_0(g_A)\|_q=q^{-\frac{1}{2q}}
   \left(\frac{\pi}{A}\right)^{\frac{1}{2p}}.
\end{align*}
We can now conclude that, when $A\to \infty$,
\begin{align}\label{asym}
\norm{\Ut (e^{-A|\xi|^2})}_q\sim q^{-\frac{s}{q}}\left(\frac{\pi}{A}\right)^{\frac{s}{p}}\cdot q^{-\frac{t}{2q}} \left(\frac{\pi}{A}\right)^{\frac{t}{2p}}.
\end{align}
Hence 
$$ \lim_{A\to \infty}\left(\frac{pA}{\pi}\right)^{\frac{d}{2p}}\norm{U_{\theta} (e^{-A\abs{\xi-s}^2})}_q=\lim_{A\to \infty}\left(\frac{pA}{\pi}\right)^{\frac{d}{2p}}\norm{U_{\theta} (e^{-A|\xi|^2})}_q=C_p^{d}.$$ 
This completes the proof.
\end{proof}

\section{A sharp Heissian comparision inequality}\label{complex-dom}
Throughout this section, let $\mathcal M$ be a von Neumann algebra equipped with a faithful normal semifinite trace $\tau$, and let $1<p<2$ with $q=p/(p-1)$. The complex Banach space
\(L_q(\mathcal M,\tau)\) is regarded as a Banach space over
\(\mathbb R\), and all Fréchet derivatives are taken with respect to real scalar parameters. The basic properties of the Fr\'echet derivative, such as the chain rule and the bilinearity of the second-order derivative, one can refer to \cite{Cartan1967}.

For \(Y\in L_q(\mathcal M,\tau)\), put 
$$\Phi(Y)=\frac1q\tau(\abs Y^q),\qquad \Omega(Y)=\tau(\abs Y^q)^{p/q}$$
which are two functions from \(L_q(\mathcal M,\tau)\) to $\R$. We abbreviate the norm of $L_q(\mathcal M,\tau)$ to $\|\cdot\|_q$. By \cite{Sukochev2019,PotapovSukochev2014}, the functional
\(\Phi(Y)\) is at least twice Fr\'echet differentiable. The same holds for $\Omega$ away from zero. Hence, we do not distinguish Fr\'echet derivatives and G\^ateaux derivatives in the sequel.

For \(B\in L_q(\mathcal M,\tau)\), the first-order derivative at $Y$ in the direction $B$,  is defined by
\[
D\Phi(Y)[B]=
 \left.
 \frac{\partial}{\partial t}
 \Phi(Y+tB)
 \right|_{t=0},
 \quad t\in\mathbb R.
 \]
The second-order derivative, for \(B,C\in L_q(\mathcal M,\tau)\), is defined by
\[
D^2\Phi(Y)[B,C]
 =
 \left.
 \frac{\partial^2}{\partial s\,\partial t}
 \Phi(Y+sB+tC)
 \right|_{s=t=0},
 \quad s,t\in\mathbb R.
\]
In particular, the second-order derivative is a bounded real bilinear form
at each fixed \(Y\), including operators with a kernel.
Since 
\[
\left.\frac{d^2}{dt^2}\Phi(Y+tB)\right|_{t=0}
=
\left.\frac1q\frac{d^2}{dt^2}
\tau\bigl(|Y+tB|^q\bigr)\right|_{t=0},
\]
for fixed \(Y\) and \(B\),
\[
\Phi(Y+tB)
=
\Phi(Y)+tD\Phi(Y)[B]
+\frac{t^2}{2}D^2\Phi(Y)[B,B]+o(t^2)
\qquad (t\to0,\ t\in\R).
\]

The following proposition is already included in \cite{Sukochev2019,PotapovSukochev2014}; we record the proof for completeness.
\begin{proposition}\label{first-formula}
With the notation above, for \(Y,B\in L_q(\mathcal M,\tau)\),
\[
 D\Phi(Y)[B]=\Rea\,\tau(\abs Y^{q-2}Y^*B).
\]
\end{proposition}
\begin{proof}
We first set
\[
A(t)=(Y+tB)^*(Y+tB).
\]
Then
\[
A(t)=Y^*Y+t(Y^*B+B^*Y)+t^2B^*B,
\]
and hence
\[
A(0)=Y^*Y,\qquad A'(0)=Y^*B+B^*Y.
\]
The derivative formula gives
\[
\left.\frac{d}{dt}\tau(A(t)^{q/2})\right|_{t=0}
=
\frac{q}{2}\,\tau\bigl(A(0)^{q/2-1}A'(0)\bigr).
\]
We can therefore compute
\begin{align*}
D\Phi(Y)[B]
&=\left.\frac{d}{dt}\Phi(Y+tB)\right|_{t=0}\\
&=\frac1q
  \left.\frac{d}{dt}\tau(A(t)^{q/2})\right|_{t=0}\\
&=\frac12\tau\left(
  (Y^*Y)^{q/2-1}(Y^*B+B^*Y)\right)\\
&=\frac12\tau\left(|Y|^{q-2}Y^*B\right)
 +\frac12\tau\left(|Y|^{q-2}B^*Y\right).
\end{align*}
Note that
\[\overline{\tau\left(|Y|^{q-2}Y^*B\right)}
=\tau\left(
  \left(|Y|^{q-2}Y^*B\right)^*\right)=\tau\left(B^*Y|Y|^{q-2}\right)=\tau\left(|Y|^{q-2}B^*Y\right).
\]
Consequently,
\[
D\Phi(Y)[B]
=
\operatorname{Re}\tau\left(|Y|^{q-2}Y^*B\right).
\]
This proves the assertion.
\end{proof}

The following lemma is a central ingredient in our proof, playing a
role analogous to that of the two-point inequality \cite[Lemma 1]{Beckner1975} in Beckner's
complex hypercontractivity argument. It compares
the second-order derivatives of $\Phi$ along $B$ and $iB$, with the sharp
coefficient $p-1$, the squared modulus of the complex parameter
$i\sqrt{p-1}$ appearing in Beckner's inequality. This comparison
holds for arbitrary, possibly noncommuting operators $Y$ and $B$.

\begin{lemma}\label{lem:traceHessian}
For \(Y,B\in L_q(\mathcal M,\tau)\),
\begin{equation}\label{eq:traceHessian}
 D^2\Phi(Y)[B,B]\ge(p-1)D^2\Phi(Y)[iB,iB].
\end{equation}
\end{lemma}

\begin{proof} Let \(A,S\in L_q(\mathcal M,\tau)\) be two self-adjoint operators, and let \(E_A\) denote the spectral measure of \(A\), such that
\[
A=\int_{\mathbb R}\lambda\,dE_A(\lambda).
\]
Then the double operator integral theory, e.g. \cite{DOIVNA,Peller2006}, gives
\[ \left.\frac{d}{dt} (A+tS)|A+tS|^{q-2}\right|_{t=0}
 =\int_{\R}\int_{\R}\delta_q(\lambda,\mu)\, d E_A(\lambda)\,S \,dE_A(\mu) \]
 where
\[
 \delta_q(x,y)=
 \begin{cases}
 \displaystyle\frac{x|x|^{q-2}-y|y|^{q-2}}{x-y},&x\ne y,\\[6pt]
 (q-1)|x|^{q-2},&x=y.
 \end{cases}
\]
In particular, \(\delta_q(0,0)=0\).

By the formula of the Fr\'echet derivatives,
$$ \frac{d^2}{dt^2}\frac1q
 \tau (|A+tS|^q)=
 \tau\left(S\cdot \frac{d}{dt}(A+tS)|A+tS|^{q-2} \right),$$
 we obtain
\begin{equation}\label{eq:infiniteTraceDerivative}
 \left.
\frac{d^2}{dt^2}\frac{1}{q}\tau\bigl(|A+tS|^q\bigr)
\right|_{t=0}
=
\tau\left(
S
\int_{\mathbb R}\int_{\mathbb R}
\delta_q(\lambda,\mu)\,
dE_A(\lambda)\,S\,dE_A(\mu)
\right).
\end{equation}
The integral in \eqref{eq:infiniteTraceDerivative} is absolutely
convergent. Indeed,
\[
 0\leq\delta_q(\lambda,\mu)
 \leq(q-1)(|\lambda|^{q-2}+|\mu|^{q-2}),
\]
and therefore
\begin{align*}
&\tau\left(S
\int_{\mathbb R}\int_{\mathbb R}
\delta_q(\lambda,\mu)\,
dE_A(\lambda)\,S\,dE_A(\mu)\right)\\
\leq\,&
(q-1)\tau\left(S
\int_{\mathbb R}\int_{\mathbb R}
(|\lambda|^{q-2}+|\mu|^{q-2})\,
dE_A(\lambda)\,S\,dE_A(\mu)\right)\\
=\,&
2(q-1)\tau\left(|A|^{q-2}S^2\right)\\
\leq\,&
2(q-1)\|A\|_q^{q-2}\|S\|_q^2
<\infty,
\end{align*}

Apply \eqref{eq:infiniteTraceDerivative} in
\[
 \widetilde{\mathcal M}=M_2(\mathbb C)\overline\otimes\mathcal M,
 \qquad \widetilde\tau=\Tr_2\otimes\tau,
\]
where $\Tr_2$ is the usual matrix trace, with
\[
 A=\begin{pmatrix}0&Y\\Y^*&0\end{pmatrix},\qquad
 S=\begin{pmatrix}0&B\\B^*&0\end{pmatrix},\qquad
 J=\begin{pmatrix}I&0\\0&-I\end{pmatrix}.
\]
Here $A,S$ are self-adjoint operators in
$L_q(\widetilde{\mathcal M},\widetilde\tau)$, and $J$ is a
self-adjoint unitary in $\widetilde{\mathcal M}$.
One easily verifies
\[
 JAJ=-A,\qquad
 \begin{pmatrix}0&iB\\-iB^*&0\end{pmatrix}=iJS,
\]
and
\[
 \Phi(Y+tB)=\frac1{2 q}\widetilde{\tau}(|A+tS|^q),\quad \Phi(Y+itB)=\frac1{2 q}\widetilde{\tau} (|A+t\cdot iJS|^q).
\]
With a similar calculation as in the proof of \eqref{eq:infiniteTraceDerivative}, we obtain
\[
 \left.\frac{d^2}{dt^2}\frac1q
 \widetilde{\tau}(|A+t\cdot iJS|^q)\right|_{t=0}
 =\widetilde{\tau}\left(S
\int_{\mathbb R}\int_{\mathbb R}
\delta_q(-\lambda,\mu)\,
dE_A(\lambda)\,S\,dE_A(\mu)\right)
\]
Hence
\begin{align}
D^2\Phi(Y)[B,B]
 &=\frac{1}{2}\widetilde{\tau}\left(S
\int_{\mathbb R}\int_{\mathbb R}
\delta_q(\lambda,\mu)\,
dE_A(\lambda)\,S\,dE_A(\mu)\right),
 \label{eq:realDirectionInfinite}\\
D^2\Phi(Y)[iB,iB]
 &=\frac{1}{2}\widetilde{\tau}\left(S
\int_{\mathbb R}\int_{\mathbb R}
\delta_q(-\lambda,\mu)\,
dE_A(\lambda)\,S\,dE_A(\mu)\right).
 \label{eq:imaginaryDirectionInfinite}
\end{align}

Put $w(\lambda,\mu)=|\lambda|^{q-2}+|\mu|^{q-2}$.
The right-hand sides of \eqref{eq:realDirectionInfinite} and
\eqref{eq:imaginaryDirectionInfinite} can be viewed as one half of
the integrals of $\delta_q(\lambda,\mu)/w(\lambda,\mu)$ and
$\delta_q(-\lambda,\mu)/w(\lambda,\mu)$ with respect to the same
finite positive Borel measure; see Remark~\ref{rem:semifiniteHessian}.
Both quotients are defined to be zero at $(0,0)$.
It remains to check the scalar inequality
\begin{equation}\label{eq:reflectedKernelComparison}
 \delta_q(\lambda ,\mu)\geq (p-1)\delta_q(-\lambda ,\mu)=\frac1{q-1}\delta_q(-\lambda ,\mu),
 \qquad \lambda ,\mu\in\mathbb R.
\end{equation}
Let \(u,v>0\). If $u=v$, then
\[
 \frac{\delta_q(u,v)}{\delta_q(-u,v)}
 =q-1.
\]
For $u\neq v$, assume \(u>v\)
and put \(a=\frac12(\log u -\log v)>0\). Then
\[
 \frac{\delta_q(u,v)}{\delta_q(-u,v)}
 =\frac{u^{q-1}-v^{q-1}}{u^{q-1}+v^{q-1}}
  \frac{u+v}{u-v}
 =\frac{\tanh((q-1)a)}{\tanh a}.
\]
Since \(q-1\geq 1\), monotonicity and concavity of ``\(\tanh\)'' give
\[
 \tanh((q-1)a)\ge\tanh a
 \ge\frac1{q-1}\tanh((q-1)a).
\]
Consequently,
\[
 \frac1{q-1}\delta_q(u,v)
 \le\delta_q(-u,v)\le\delta_q(u,v).
\]
Using
\(\delta_q(-\lambda,-\mu)=\delta_q(\lambda ,\mu)\) and \(\delta_q(\lambda ,\mu)=\delta_q(\mu,\lambda)\), these two inequalities give
\[
 \delta_q(\lambda ,\mu)\ge\frac1{q-1}\delta_q(-\lambda ,\mu).
\]
This
proves \eqref{eq:reflectedKernelComparison}.

Finally, comparing each entry in \eqref{eq:realDirectionInfinite} and
\eqref{eq:imaginaryDirectionInfinite}
 proves the desired result.
\end{proof}
\begin{remark}\label{rem:semifiniteHessian}
Since $S\in L_q(\widetilde{\mathcal M},\widetilde\tau)$,
$\|E_A(\Gamma_1)SE_A(\Gamma_2)\|_2^2$ is possibly infinite.
However, by H\"older's inequality,
\[
 K_1=|A|^{(q-2)/2}S,\qquad K_2=S|A|^{(q-2)/2}
 \quad\hbox{belong to }L_2(\widetilde{\mathcal M},\widetilde\tau).
\]
The commuting left and right spectral projections of $A$ on this
$L_2$-space define finite positive Borel measures $v_j$ by
\[
 v_j(\Gamma_1\times\Gamma_2)
 =\|E_A(\Gamma_1)K_jE_A(\Gamma_2)\|_2^2,
 \qquad j=1,2.
\]
Then $v'=v_1+v_2$ is the required measure, and
\eqref{eq:realDirectionInfinite} and
\eqref{eq:imaginaryDirectionInfinite} are one half of the integrals of
\[
 \frac{\delta_q(\lambda,\mu)}{|\lambda|^{q-2}+|\mu|^{q-2}},
 \qquad
 \frac{\delta_q(-\lambda,\mu)}{|\lambda|^{q-2}+|\mu|^{q-2}}
\]
against $v'$. Thus the scalar comparison applies also when
$\widetilde\tau(I)=\infty$.
\end{remark}

\Needspace{8\baselineskip}
\begin{lemma}\label{lem:normHessian}
For every nonzero
\(Y\in L_q(\mathcal M,\tau)\) and every \(B\in L_q(\mathcal M,\tau)\), let \(\Omega(Y)=\norm Y_q^p\).  Then 
\begin{equation}\label{eq:normHessian}
 D^2\Omega(Y)[B,B]\ge(p-1)D^2\Omega(Y)[iB,iB].
\end{equation}
\end{lemma}

\begin{proof}
Fix \(Y\ne0\), for real \(t,s\), the definition of the second derivative gives
\[
 D^2\Phi(Y)[Y,B]
 =
 \left.\frac{\partial}{\partial t}\right|_{t=0}
 \left.\frac{\partial}{\partial s}\right|_{s=0}
 \Phi\bigl((1+t)Y+sB\bigr).
\]
Applying the derivative formula and Proposition~\ref{first-formula}, we obtain
\[
 \left.\frac{\partial}{\partial s}
 \Phi\bigl((1+t)Y+sB\bigr)\right|_{s=0}
 =
 \operatorname{Re}\tau
 \bigl(|(1+t)Y|^{q-2}((1+t)Y)^*B\bigr).
\]
Consequently,
\[
 \left.\frac{\partial}{\partial s}
 \Phi\bigl((1+t)Y+sB\bigr)\right|_{s=0}
 =
 (1+t)^{q-1}
 \operatorname{Re}\tau(|Y|^{q-2}Y^*B).
\]
Differentiating with respect to \(t\) at \(t=0\) gives
\[
 \begin{aligned}
 D^2\Phi(Y)[Y,B]=
 \left.\frac{d}{dt}(1+t)^{q-1}\right|_{t=0}\cdot
 \operatorname{Re}\tau(|Y|^{q-2}Y^*B)=(q-1)D\Phi(Y)[B].
 \end{aligned}
\]
Taking $Y=B$ gives 
$$ D^2\Phi(Y)[Y,Y]=(q-1)\tau(\abs Y^q).$$

Again by the definition of the second-order derivative,
\[
 D^2\Phi(Y)[iY,iB]
 =
 \left.\frac{d}{dt}D\Phi((1+it)Y)[iB]\right|_{t=0},
 \qquad t\in\mathbb R.
\]
Applying Proposition~\ref{first-formula} again gives
\begin{align*}
 D\Phi((1+it)Y)[iB]
 &=
 \operatorname{Re}\tau
 \bigl(|(1+it)Y|^{q-2}((1+it)Y)^*iB\bigr)\\
 &=(1+t^2)^{(q-2)/2}
   \operatorname{Re}\Bigl((i+t)
   \tau(|Y|^{q-2}Y^*B)\Bigr).
\end{align*}
Differentiating the preceding
identity therefore yields
\begin{align*}
 D^2\Phi(Y)[iY,iB]&=
 \left.\frac{d}{dt}D\Phi((1+it)Y)[iB]\right|_{t=0}\\
 &=\operatorname{Re}\Bigl(\tau(|Y|^{q-2}Y^*B)\Bigr)\\
 &=D\Phi(Y)[B].
\end{align*}
Taking $Y=B$ gives $$ D^2\Phi(Y)[iY,iY]=\tau(\abs Y^q).$$

For convenience, we abbreviate
\[
 \alpha:=\frac{D\Phi(Y)[B]}{q\Phi(Y)},\qquad
 \beta:=\frac{D\Phi(Y)[iB]}{q\Phi(Y)},
\]
and
\[
 \Delta_Y(B):=D^2\Phi(Y)[B,B]-(p-1)D^2\Phi(Y)[iB,iB].
\]
Lemma \ref{lem:traceHessian} gives \(
 \Delta_Y(B)\geq 0 \).
The bilinearity of the second-order Fr\'echet derivative gives
\begin{equation}
\begin{aligned}
 \Delta_Y(B-\alpha Y)=&D^2\Phi(Y)[B-\alpha Y,B-\alpha Y]-(p-1)D^2\Phi(Y)[i(B-\alpha Y),i(B-\alpha Y)]\\
 =& D^2\Phi(Y)[B,B]-2\alpha D^2\Phi(Y)[B,Y]+\alpha^2D^2\Phi(Y)[Y,Y]\\
 &-(p-1)\Big(D^2\Phi(Y)[iB,iB]-2\alpha D^2\Phi(Y)[iB,iY]+\alpha^2D^2\Phi(Y)[iY,iY]\Big)\\
 =&\Delta_Y(B)-2\alpha\Big(D^2\Phi(Y)[B,Y]-(p-1)D^2\Phi(Y)[iB,iY]\Big)\\
 &+\alpha^2\Big(D^2\Phi(Y)[Y,Y]-(p-1)D^2\Phi(Y)[iY,iY]\Big)\\
  =&\Delta_Y(B)-2\alpha\Big((q-1)D\Phi(Y)[B]-(p-1)D\Phi(Y)[B]\Big)\\
 &+\alpha^2\Big((q-1)\tau(\abs Y^q)-(p-1)\tau(\abs Y^q)\Big)\\
 =&\Delta_Y(B)-(q-p)\alpha^2\tau(\abs Y^q).
\end{aligned}\label{eq:radialSquare}
\end{equation}
On the other hand, 
\[\Omega(Y)=\tau(\abs Y^q)^{p/q}=(q\Phi(Y))^{p/q}.\]
Denoting $h(x)=|x|^{p/q}$, the chain rule of the Fr\'echet derivative gives
\[D^2\Omega(Y)[B,B]=h^{\prime\prime}(q\Phi(Y))\big(qD\Phi(Y)[B]\big)^2+qh^{\prime}(q\Phi(Y))D^2\Phi(Y)[B,B].
\]
Hence
\[
 \frac{D^2\Omega(Y)[B,B]}{p\Omega(Y)}
 =\frac{D^2\Phi(Y)[B,B]}{q\Phi(Y)}-(q-p)\alpha^2.
\]
Similarly we get
\[
 \frac{D^2\Omega(Y)[iB,iB]}{p\Omega(Y)}
 =\frac{D^2\Phi(Y)[iB,iB]}{q\Phi(Y)}-(q-p)\beta^2.
\]
Combining these two formulae and
\eqref{eq:radialSquare}, we obtain
\begin{align*}
 \frac{D^2\Omega(Y)[B,B]-(p-1)D^2\Omega(Y)[iB,iB]}
 {p\Omega(Y)}
 =\frac{\Delta_Y(B-\alpha Y)}{q\Phi(Y)}+(p-1)(q-p)\beta^2.
\end{align*}
The desired assertion then follows from the nonnegativity of $\Delta_Y(B-\alpha Y)$.
\end{proof}

\section{Monotonicity of the flow}\label{sec-main}

Throughout the monotonicity argument, let $1<p<2$ and $q=p/(p-1)$.
Recall that, for $A>0$, we define
\[
 \mathcal L_A(f)=\left(\frac{pA}{\pi}\right)^{\frac{d}{2}}\int_{\R^d}
  \norm{\Ut\bigl(f(\cdot) e^{-A\abs{\,\cdot-s\,}^2}\bigr)}_{\LpQ q}^{p}
 \,ds.
\]
In this section, we are going to prove its monotonicity, which will yield the sharp Hausdorff--Young inequality directly.
\subsection{Auxiliary results}
Fix a nonzero complex-valued \(f\in C_c^\infty(\R^d)\). For
\(a\in\R_+\), \(b\in\R^d\), set
\[
 f_{a,b}(\xi)=f(\xi)e^{-a\abs\xi^2+b\cdot \xi}\]
 and put
 \begin{equation}\label{eq:VI}
 V(a,b)=\norm{\Ut (f_{a,b})}_q^p,\qquad
 I(a,b)=\int_{\R^d}\abs{f_{a,b}(\xi)}^p\,d\xi.
\end{equation}
The non-sharp Hausdorff-Young inequality directly gives
\begin{equation}\label{eq:VIbound}
 0<V(a,b)\le I(a,b).
\end{equation}
The smoothness is trivial for \((a,b)\mapsto I(a,b)\). Moreover, the map \((a,b)\mapsto f_{a,b}\) is smooth in \(L_p(\R^d)\), which gives \((a,b)\mapsto \Ut(f_{a,b})\) smooth in $\LpQ{q}$ by the non-sharp Hausdorff-Young inequality. Hence there is no ambiguity when taking derivatives. 

 In the following proposition, we make the convention that
 $$ \Delta_b V(a,b)=(\partial_{b_1}^2+\cdots+\partial_{b_d}^2)V(a,b),\quad \partial_1 V=\partial_a V(a,b).$$
\begin{proposition}\label{prop:heatInequality}
With the notation above,
\begin{equation}\label{eq:heatInequality}
 (\Delta_b+p\partial_1)V\ge0.
\end{equation}
\end{proposition}

\begin{proof}
At a fixed \((a,b)\), write
\[
Y=\Ut (f_{a,b}),\qquad
 X_j=\Ut(\xi_j f_{a,b}),\qquad
 C_j=\Ut(\xi_j^2 f_{a,b}),\quad j=1,2,\cdots,d.
\]
Differentiation in the parameters gives
\[
  \partial_{b_j}Y=X_j,\quad \partial_{b_j}^2Y=C_j,
 \quad \partial_aY=-\sum_{j=1}^{d}C_j.
\]
Since translation  is unitary, the map
\[t\longmapsto\norm{\Ut(e^{it\xi_j}f_{a,b})}_q^p=\Omega(\Ut(e^{it\xi_j}f_{a,b}))\]
is constant. Its second derivative at zero is therefore
\begin{align*}
 \frac{d^2}{dt^2}\Omega(\Ut(e^{it\xi_j}f_{a,b}))|_{t=0}&=D^2\Omega(\Ut(f_{a,b}))[i\Ut(\xi_j f_{a,b}),i\Ut(\xi_j f_{a,b})]\\
 &\hspace{15mm}-D\Omega(\Ut(f_{a,b}))[\Ut(\xi_j^2f_{a,b})]\\
 &=D^2\Omega(Y)[iX_j,iX_j]-D\Omega(Y)[C_j]\\
 &=0.
\end{align*}
This yields 
$$D^2\Omega(Y)[iX_j,iX_j]=D\Omega(Y)[C_j].$$
By the second-order chain rule,
\[
 \partial_{b_j}^2V
 =D^2\Omega(Y)[X_j,X_j]+D\Omega(Y)[C_j].
\]
Lemma~\ref{lem:normHessian} then gives
\[
 \partial_{b_j}^2V
 \ge (p-1)D^2\Omega(Y)[iX_j,iX_j]
     +D\Omega(Y)[C_j]=pD\Omega(Y)[C_j].
\]
Summing over \(j\) proves the inequality for \(V(a,b)\).
\end{proof}

Let \(H_t\) denote the ordinary heat operator semigroup:
\begin{equation}\label{eq:heatKernel}
 H_t g(x)=\left(\frac1{4\pi t}\right)^{\frac{d}{2}}\int_{\R^d}
 e^{-\abs{x-y}^2/(4t)}g(y)\,dy,\qquad t>0,
\end{equation}
and \(H_0g=g\). In the following, we will frequently use the well known positivity and semigroup identity ($H_s H_t = H_{s+t}$) of the heat semigroup.

\begin{proposition}\label{prop:comparison}
For \(0\le a\le A\),
\begin{equation}\label{eq:comparison}
 V(a,b)\le H_{(A-a)/p}\bigl[V(A,\cdot)\bigr](b).
\end{equation}
\end{proposition}

\begin{proof}
The comparison follows directly by integrating
\eqref{eq:heatInequality} along the heat flow. Fix \(0\le a \leq A\)
and \(b\in\R^d\). For \(0\le s\le A-a\), view $V(a+s,\cdot)$ as one variable function, and consider the function
\[
H_{s/p}[V(a+s,\cdot)](b)= \left(\frac{p}{4\pi s}\right)^{\frac{d}{2}}\int_{\R^d}
 e^{-\abs{b-y}^2p/(4s)}V(a+s,y)\,dy.
\]
For \(0\leq s \leq A-a\), differentiation under the integral and integration
by parts give
\begin{align*}
\frac{d}{ds}H_{s/p}[V(a+s,\cdot)](b)
&=
\frac1p H_{s/p}[\Delta_bV(a+s,\cdot)](b)+
H_{s/p}[\partial_1V(a+s,\cdot)](b)\\
&=
\frac1p H_{s/p}
 \bigl[(\Delta_b+p\partial_1)V(a+s,\cdot)\bigr](b).
\end{align*}
By Proposition~\ref{prop:heatInequality}, $(\Delta_b+p\partial_1)V$ is nonnegative. Since the heat operator preserves
nonnegativity,
\[
\frac{d}{ds}H_{s/p}[V(a+s,\cdot)](b)\ge0.
\]
Consequently,
\[
V(a,b)
=H_0[V(a,\cdot)](b)
\le
H_{(A-a)/p}[V(A,\cdot)](b).
\]
The case \(a=A\) is immediate.

To verify the convergence of the heat integrals, choose \(R>0\) such that
\(\supp f\subset B(0,R)\). Then, uniformly for \(a\ge0\),
\[
0<V(a,b)\le I(a,b)
\le \|f\|_p^p e^{pR|b|}.
\]
This exponential growth is dominated by the Gaussian heat kernel
and its derivatives, which justifies the preceding differentiation
and integration by parts.
\end{proof}

\subsection{Main result}
\begin{theorem}\label{thm:monotonicity}
For every complex-valued \(f\in L_p(\R^d)\), the quantity
\(\mathcal L_A(f)\) in \eqref{eq:localization} is finite and
nondecreasing for \(A> 0\). Moreover,
\begin{equation}\label{eq:limits}
 \lim_{A\to 0^+}\mathcal L_A(f)=\norm{\Ut (f)}_q^p,
 \qquad
 \lim_{A\to\infty}\mathcal L_A(f)=C_p^{dp}\norm f_p^p.
\end{equation}
\end{theorem}

\begin{proof}
First we record a uniform continuity estimate. For fixed \(A>0\),
the map
\[
 f\longmapsto
 \left(\frac{pA}{\pi}\right)^{\frac{d}{2p}}
 \Ut\bigl(fe^{-A\abs{\,\cdot-s\,}^2}\bigr)
\]
is a contraction from \(L_p(\R^d)\) into
\(L_p(\R^d;\LpQ q)\). Indeed, non-sharp Hausdorff-Young inequality gives
\begin{align*}
 \mathcal L_A(f)&= \left(\frac{pA}{\pi}\right)^{\frac{d}{2}}\int_{\R^d}
  \norm{\Ut\bigl(f(\,\cdot\,) e^{-A\abs{\,\cdot-s\,}^2}\bigr)}_{\LpQ q}^{p}
 \,ds\\
 &\le
 \left(\frac{pA}{\pi}\right)^{\frac{d}{2}}\int_{\R^d}\int_{\R^d}
 \abs{f(\xi)}^pe^{-pA\abs{\xi-s}^2}\,d\xi\,ds\\
 &=\norm f_p^p.
\end{align*}
Consider the map
\begin{align*}
 T_A: L_p(\R^d)&\longrightarrow L_p(\R^d;L_q(\R_\theta^d))\\
                f(\xi)    \,\,\,&\longmapsto F(s).
\end{align*}
where $F(s)=\left(\frac{pA}{\pi}\right)^{\frac{d}{2p}}\Ut\bigl(f(\cdot) e^{-A\abs{\,\cdot-s\,}^2}\bigr)$. In this sense, $\mathcal L_A(f)^{1/p}$ is actually the norm of $T_A(f)$  in $L_p(\R^d;\LpQ{q})$. Hence $T_A$ is a contraction from $L_p(\R^d)$ to $L_p(\R^d;\LpQ{q})$.  This yields, uniformly in \(A\),
\begin{equation}\label{eq:Lipschitz}
  \abs{\mathcal L_A(f)^{1/p}-\mathcal L_A(h)^{1/p}}\leq \mathcal L_A(f-h)^{1/p}
 \le\norm{f-h}_p.
\end{equation}

We prove the monotonicity first. Take nonzero \(f\in C_c^\infty(\R^d)\).
By the construction of $f_{a,b}$, we have
\[
 e^{-A\abs{s}^2}f_{A,2A s}(\xi)
 =f(\xi)e^{-A\abs{\xi-s}^2}.
\]
Therefore,
\[
\left\|
U_\theta\bigl(f(\xi)e^{-A\abs{\xi-s}^2}\bigr)
\right\|_q^p
=
e^{-pA|s|^2}V(A,2As).
\]
Substituting this identity into the definition of
\(\mathcal L_A(f)\), we obtain
\[
\mathcal L_A(f)
=
\left(\frac{pA}{\pi}\right)^{\frac{d}{2}}
\int_{\mathbb R^d}
e^{-pA|s|^2}V(A,2As)\,ds.
\]
The change of variables \(x=2As\) gives
\begin{align}\label{eq:Lheat}
\mathcal L_A(f)=
\left(\frac{p}{4\pi A}\right)^{\frac{d}{2}}
\int_{\mathbb R^d}
e^{-p|x|^2/(4A)}V(A,x)\,dx=
H_{A/p}[V(A,\cdot)](0).
\end{align}
If \(0<A<B\), Proposition \ref{prop:comparison} yields
\[V(A,\cdot)\le H_{(B-A)/p}V(B,\cdot).\]
Positivity and the
semigroup identity for \(H_t\) imply
\[
 \mathcal L_A(f)=H_{A/p}[V(A,\cdot)](0)
 \le H_{A/p}H_{(B-A)/p}[V(B,\cdot)](0)
 =\mathcal L_B(f).
\]
Estimate \eqref{eq:Lipschitz}
and the density of \(C_c^\infty\) in \(L_p\) extend this inequality
to every \(f\in L_p(\R^d)\).

Next consider \(A\to 0^+\), again first with \(f\in C_c^\infty(\R^d)\). Rewrite \eqref{eq:Lheat} as
\[
 \mathcal L_A(f)=\left(\frac1\pi\right)^{\frac{d}{2}}\int_{\R^d}e^{-\abs x^2}
 V\left(A,2\sqrt{\frac{A}{p}}\,x\right)\,dx.
\]
The integrand converges pointwise to \(e^{-\abs x^2}V(0,0)\).
For \(0<A\le1\) and \(\supp f\subset B(0,R)\), we have
\[
 V\left(A,2\sqrt{\frac{A}{p}}\,x\right)\leq \int_{\R^d}|f(\xi)|^p e^{-pA|\xi|^2+2\sqrt{pA}x\cdot \xi}\,d\xi
\leq \norm f_p^p e^{2R\sqrt{p}|x|}.
\]
which the right-hand side, if multiplied by $e^{-\abs x^2}$, is an integrable function. Dominated convergence gives
$$\lim_{A\to 0^+} \mathcal L_A(f) = \left(\frac1\pi\right)^{\frac{d}{2}}V(0,0)\cdot\int_{\R^d}e^{-\abs x^2}\,dx=V(0,0)=\|\Ut(f)\|_q^p.$$
Density argument passes the result to all $f\in L_p(\R^d)$.

The limit as \(A\to\infty\) can be proved directly for every
\(f\in L_p(\R^d)\). Note that
\[
 R_A(s)=\left(\frac{pA}{\pi}\right)^{\frac{d}{2p}}
 \norm{\Ut(fe^{-A\abs{\,\cdot\,-s}^2})}_q,
 \qquad
 d_A=\left(\frac{pA}{\pi}\right)^{\frac{d}{2p}}\norm{\Ut (e^{-A\abs{\,\cdot\,-s}^2})}_q.
\]
Recall that
$$ \Ut (f(\cdot -s))=\Ut(s/2)\Ut(f)\Ut(s/2). $$
Together with Proposition~\ref{lem:gaussian},
\[\lim_{A\to \infty}d_A=\lim_{A\to \infty}\left(\frac{pA}{\pi}\right)^{\frac{d}{2p}}\norm{\Ut(e^{-A\abs{\,\cdot\,-s}^2})}_q=C_p^d.\]
The reverse triangle inequality thus gives,
for almost every \(s\),
\[
 \Big|R_A(s)-d_A\abs{f(s)}\Big|^p\leq\left(\frac{pA}{\pi}\right)^{\frac{d}{2}}\|\Ut(fe^{-A\abs{\,\cdot\,-s}^2})-f(s)\Ut(e^{-A\abs{\,\cdot\,-s}^2})\|_q^p.\]
Then by the non-sharp Hausdorff-Young inequality, we have
\[
 \Big|R_A(s)-d_A\abs{f(s)}\Big|^p\le\left(\frac{pA}{\pi}\right)^{\frac{d}{2}}\int_{\R^d}
 \abs{f(\xi)-f(s)}^p e^{-pA\abs{\xi-s}^2}\,d\xi.
\]
Integration in \(s\) with $v=\xi-s$ gives,
\begin{equation}\label{eq:approximation}
 \Big\|R_A-d_A\abs f \Big\|_p^p
 \le\left(\frac{pA}{\pi}\right)^{\frac{d}{2}}\int_{\R^d}e^{-pA\abs v^2}
 \norm{f(\cdot+v)-f}_p^p\,dv\longrightarrow0.
\end{equation}
Consequently
\(R_A\to C_p^d\abs f\) in \(L^p\). Since
\(\norm{R_A}_p^p=\mathcal L_A(f)\), this proves the second limit.
\end{proof}

Now we are able to prove the sharp Hausdorff--Young inequality and its dual version on quantum euclidean spaces.
\begin{theorem}\label{main-HY}
Let $1\le p\le2$ and $q=p'$. For every \(f\in L_p(\R^d)\) and every \(x\in L_p(\R_\theta^d) \),
\[
 \norm{\Ut (f)}_{L_q(\R_\theta^d)}\le C_p^d\norm f_{L_p(\R^d)},
 \qquad
 \norm{\Ft (x)}_{L_q(\R^d)}\le C_p^d\norm x_{L_p(\R_\theta^d)}.
\]
The constant \(C_p^d\) is sharp in both inequalities.
\end{theorem}
\begin{proof}
The cases $p=1$ and $p=2$ follow from the integral bound and Plancherel,
respectively, and their constants are sharp. Let $1<p<2$.
The first inequality follows from monotonicity and the two limits
in Theorem \ref{thm:monotonicity}. The Gaussians of Proposition~\ref{lem:gaussian} asymptotically attain \(C_p^d\), so this
constant is sharp.

For the second inequality, first take
\(x\in L_1(\mathbb R_\theta^d)\cap L_p(\mathbb R_\theta^d)\).
For every
\(f\in L_1(\mathbb R^d)\cap L_p(\mathbb R^d)\),
the basic properties of quantum euclidean space yield
\[
\int_{\mathbb R^d}
\mathcal F_\theta(x)(\xi)\overline{f(\xi)}\,d\xi
=
\tau_\theta\bigl(xU_\theta(f)^*\bigr).
\]
By noncommutative H\"older's inequality and the proved inequality
for $U_\theta$, one obtains
\[
\left|
\tau_\theta\bigl(xU_\theta(f)^*\bigr)
\right|
\le
\|x\|_{L_p(\mathbb R_\theta^d)}
\|U_\theta(f)\|_{L_q(\mathbb R_\theta^d)}
\le
C_p^d
\|x\|_{L_p(\mathbb R_\theta^d)}
\|f\|_{L_p(\mathbb R^d)}.
\]
Thus the scalar \(L_p\)-duality implies that
\(\mathcal F_\theta(x)\in L_q(\mathbb R^d)\).
Using this duality and the density of
\(L_1(\mathbb R^d)\cap L_p(\mathbb R^d)\) in
\(L_p(\mathbb R^d)\), we obtain
\[
\begin{aligned}
\|\mathcal F_\theta(x)\|_{L_q(\mathbb R^d)}
&=
\sup_{\substack{
f\in L_1\cap L_p\\
\|f\|_{p}=1
}}
\left|
\int_{\mathbb R^d}
\mathcal F_\theta(x)(\xi)\overline{f(\xi)}\,d\xi
\right|\\
&=
\sup_{\substack{
f\in L_1\cap L_p\\
\|f\|_{p}=1
}}
\left|
\tau_\theta\bigl(xU_\theta(f)^*\bigr)
\right|\\
&\le
C_p^d
\sup_{\substack{
y\in L_q(\mathbb R_\theta^d)\\
\|y\|_{q}=1
}}
\left|\tau_\theta(xy^*)\right|\\
&=
C_p^d\|x\|_{L_p(\mathbb R_\theta^d)}.
\end{aligned}
\]
Finally, the density of
\(L_1(\mathbb R_\theta^d)\cap L_p(\mathbb R_\theta^d)\)
in \(L_p(\mathbb R_\theta^d)\) yields the second
inequality for every \(x\in L_p(\mathbb R_\theta^d)\). Take $x_t=\Ut(e^{-t|\cdot|^2})$. Following the computation in Proposition
\ref{lem:gaussian}, we obtain
$$ \lim_{t\to \infty}\frac{\|e^{-t|\cdot|^2}\|_q}{\|x_t\|_p}=C_p^d.$$
This proves the sharpness for the second inequality.
\end{proof}


\section{Heisenberg
groups}\label{sec-heisenberg}
We now prove the sharp Hausdorff--Young inequality on Heisenberg
groups by the same flow trick. We first recall its basic properties; see
\cite{Folland1989,KleinRusso1978,CowlingEtAl2019}.

\subsection{Harmonic analysis on Heisenberg
groups}
Recall $\mathbb H_n= \mathbb C^n\times\mathbb R$ with the group law
\(
 (z,t)(z',t')
 =\big(z+z',t+t'+ \operatorname{Im}(\overline z\cdot z')/2\big)
\) and Haar measure $dz\,dt$. The Schr\"odinger representation $\pi_\lambda$ acts on
$L_2(\mathbb R^n)$ by
\[
 \bigl(\pi_\lambda(u+iv,t)\varphi\bigr)(x)
 =e^{2\pi i\lambda t
    +2\pi i\operatorname{sgn}(\lambda)\sqrt{|\lambda|}\,v\cdot x
    +\pi i\lambda u\cdot v}
 \varphi\bigl(x+\sqrt{|\lambda|}\,u\bigr),\qquad z=u+iv.
\]
For $f\in L_1(\mathbb H_n)$, define
\[
 \mathcal{F}_{\mathbb H_n}(f)(\lambda)=\pi_\lambda(f)
 =\int_{\mathbb C^n}\int_{\mathbb R}
 f(z,t)\pi_\lambda(z,t)\,dt\,dz.
\]
The group von Neumann algebra $\mathcal M_{\mathbb H_n}$ is identified
with the direct integral of $\mathcal B(L_2(\mathbb R^n))$ over
$\mathbb R\setminus\{0\}$, with the Plancherel trace
\[
 \tau_{\mathbb H_n}(X)
 =\int_{\mathbb R\setminus\{0\}}
 \Tr(X(\lambda))|\lambda|^n\,d\lambda,
 \qquad X\ge0.
\]
For $1\le r<\infty$, put
\[
 \|X\|_r
 =\left(\int_{\mathbb R\setminus\{0\}}
 \|X(\lambda)\|_{\mathcal S_r}^r|\lambda|^n\,d\lambda\right)^{1/r},
\]
with the essential supremum of the operator norms when $r=\infty$.
The integral bound and Plancherel formula give
\[
 \|\mathcal{F}_{\mathbb H_n}(f)\|_\infty\le\|f\|_1,
 \qquad \|\mathcal{F}_{\mathbb H_n}(f)\|_2=\|f\|_2.
\]
Interpolation yields the non-sharp Hausdorff--Young inequality
\begin{equation}\label{h:nonsharp}
 \|\mathcal{F}_{\mathbb H_n}(f)\|_q\le\|f\|_p.
\end{equation}
In particular, $f\mapsto \mathcal{F}_{\mathbb H_n}(f)$ extends uniquely to a bounded map
from $L_p(\mathbb H_n)$ to $L_q(\mathcal M_{\mathbb H_n},\tau_{\mathbb H_n})$.

We need to introduce two isometries. Identify $\mathbb C^n$ with
$\mathbb R^{2n}$. The trace preserving automorphisms
\begin{equation}\label{h:phase}
 T_r(\pi_\lambda(z,t))
 =e^{ir\cdot z}\pi_\lambda(z,t),\qquad r,z\in\mathbb R^{2n},
\end{equation}
act isometrically on each $L_q(\mathcal M_{\mathbb H_n},\tau_{\mathbb H_n})$
and satisfy
\[
 \mathcal{F}_{\mathbb H_n}(e^{ir\cdot z}f)=T_r(\mathcal{F}_{\mathbb H_n}(f)).
\]
Indeed, $(z,t)\mapsto e^{ir\cdot z}$ is a character of $\mathbb H_n$.
Moreover, the group law gives
\[
 (s/2,0)(z,t)(s/2,0)=(z+s,t).
\]
Consequently,
\begin{equation}\label{h:translation}
 \mathcal{F}_{\mathbb H_n}\big(f(\,\cdot-s,\cdot\,)\big)(\lambda)
 =\pi_\lambda(s/2,0)\circ \mathcal{F}_{\mathbb H_n}(f)(\lambda)\circ\pi_\lambda(s/2,0).
\end{equation}
Since $\pi_\lambda(s/2,0)$ is unitary, the two-sided multiplication
in \eqref{h:translation} is an $L_p$-isometry.

The unitaries $\pi_\lambda(z,0)$ satisfy the Weyl relations
\[
 \pi_\lambda(z,0)\pi_\lambda(z',0)
 =e^{\pi i\lambda\operatorname{Im}(\overline z\cdot z')}
 \pi_\lambda(z+z',0).
\]
They give a realization of the quantum euclidean space associated with
\[
 \theta_\lambda=2\pi\lambda
 \begin{pmatrix}\mathbf 0&I_n\\-I_n&\mathbf 0\end{pmatrix},
 \qquad \tau_{\theta_\lambda}=|\lambda|^n\Tr.
\]

For later use, write
\[
\widehat f_t(z,\xi)
=\int_{\mathbb R}f(z,t)e^{-it\xi}\,dt,
\qquad
d\mu(\xi)=\frac{d\xi}{2\pi}.
\]
Also put
\[
W_\lambda(g)=\int_{\mathbb C^n}g(z)\pi_\lambda(z,0)\,dz.
\]
Thus $\mathcal{F}_{\mathbb H_n}(f)(\lambda)=W_\lambda\bigl(\widehat f_t(\cdot,-2\pi\lambda)\bigr)$. Note that, writing $z=u+iv$, we have
$$W_\lambda(f(z))=U_{\theta_\lambda}(f(u+iv)),\qquad |\lambda|^{n/p}\|W_\lambda(f(z))\|_{\mathcal{S}_p}=\|U_{\theta_\lambda}(f(u+iv))\|_p.$$
Moreover, put
\[
M(f)=
\left[
\int_{\mathbb C^n}
\left(
\int_{\mathbb R}
|\widehat f_t(z,\xi)|^q\,d\mu(\xi)
\right)^{p/q}\,dz
\right]^{1/p}.
\]
The classical Hausdorff--Young inequality gives
$M(f)\le\|f\|_p$.

For $A>0$, define
\begin{equation}\label{h:localization}
 \mathcal L_A^{\mathbb H_n}(f)
 =\left(\frac{pA}{\pi}\right)^n\int_{\mathbb C^n}
 \left\|\mathcal{F}_{\mathbb H_n}\big(f(z,t)e^{-A|z-s|^2}\big)\right\|_q^p\,ds.
\end{equation}
In the following two subsections we assume $1<p<2$.
We are going to prove the monotonicity of $ \mathcal L_A^{\mathbb H_n}(f)$
and compute its endpoint limits.

\medskip

Fix a nonzero complex-valued $f\in C_c^\infty(\mathbb H_n)$.
For $a\in\mathbb R_+$, $b\in\mathbb R^{2n}$, set
\[
 f_{a,b}(z,t)=f(z,t)e^{-a|z|^2+b\cdot z}
\]
and put
\begin{equation}\label{h:VI}
 V(a,b)=\|\mathcal{F}_{\mathbb H_n}(f_{a,b})\|_q^p,
 \qquad I(a,b)=\int_{\mathbb H_n}|f_{a,b}(z,t)|^p\,dz\,dt.
\end{equation}
The non-sharp Hausdorff--Young inequality directly gives
\begin{equation}\label{h:VIbound}
 0<V(a,b)\le I(a,b).
\end{equation}
The map $(a,b)\mapsto f_{a,b}$ is smooth in $L_p(\mathbb H_n)$,
which gives $(a,b)\mapsto\mathcal{F}_{\mathbb H_n}(f_{a,b})$ smooth in
$L_q(\mathcal M_{\mathbb H_n},\tau_{\mathbb H_n})$ by the non-sharp Hausdorff-Young inequality. Hence there is no ambiguity when taking derivatives.
As in Section~\ref{sec-main}, write
\[
 \Delta_bV=(\partial_{b_1}^2+\cdots+\partial_{b_{2n}}^2)V,
 \qquad \partial_1V=\partial_aV.
\]

\begin{proposition}\label{h:heatInequality}
With the notation above,
\[
 (\Delta_b+p\partial_1)V\ge0.
\]
\end{proposition}

\begin{proof}
At a fixed $(a,b)$, write
\[
 Y=\mathcal{F}_{\mathbb H_n}(f_{a,b}),\qquad
 X_j=\mathcal{F}_{\mathbb H_n}(z_jf_{a,b}),\qquad
 C_j=\mathcal{F}_{\mathbb H_n}(z_j^2f_{a,b}).
\]
Observe that, by \eqref{h:phase}, the map
\[
 r\longmapsto\left\|\mathcal{F}_{\mathbb H_n}(e^{irz_j}f_{a,b})\right\|_q^p
 =\Omega\left(\mathcal{F}_{\mathbb H_n}(e^{irz_j}f_{a,b})\right)
\]
is constant. Using this fact and repeating the argument for Proposition \ref{prop:heatInequality} give the desired inequality. 
\end{proof}

Let $H_t$ denote the ordinary heat operator in
\eqref{eq:heatKernel}, now acting on $\mathbb R^{2n}$.
Following a similar argument as Proposition~\ref{prop:comparison}, we obtain
\begin{equation}\label{h:comparison}
 V(a,b)\le H_{(A-a)/p}[V(A,\cdot)](b),\qquad 0\le a\le A.
\end{equation}

\subsection{Main result on Heisenberg groups}

\begin{theorem}\label{h:monotonicity}
For every complex-valued $f\in L_p(\mathbb H_n)$, the quantity
$\mathcal L_A^{\mathbb H_n}(f)$ in \eqref{h:localization} is finite
and non-decreasing for $A>0$. Moreover,
\begin{equation}\label{h:limits}
 \lim_{A\to0^+}\mathcal L_A^{\mathbb H_n}(f)=\|\mathcal{F}_{\mathbb H_n}(f)\|_q^p,
 \qquad
 \lim_{A\to\infty}\mathcal L_A^{\mathbb H_n}(f)=C_p^{2np}M(f)^p.
\end{equation}
\end{theorem}

\begin{proof}
Following a similar argument in Theorem \ref{thm:monotonicity}, one easily get, uniformly in $A$,
\begin{equation}\label{h:Lipschitz}
 \left|\mathcal L_A^{\mathbb H_n}(f)^{1/p}
       -\mathcal L_A^{\mathbb H_n}(g)^{1/p}\right|
 \le\|f-g\|_p
\end{equation}
by the non-sharp Hausdorff--Young inequality. Hence we can prove the result densely for $f\in C_c^\infty(\mathbb H_n)$.

We prove the monotonicity first. By the construction of $f_{a,b}$,
\[
 e^{-A|s|^2}f_{A,2As}(z,t)=f(z,t)e^{-A|z-s|^2}.
\]
Substituting this identity into the definition of
$\mathcal L_A^{\mathbb H_n}(f)$ gives
\begin{align}\label{h:Lheat}
 \mathcal L_A^{\mathbb H_n}(f)
 &=\left(\frac{pA}{\pi}\right)^n
 \int_{\mathbb R^{2n}}e^{-pA|s|^2}V(A,2As)\,ds\notag\\
 &=\left(\frac{p}{4\pi A}\right)^n
 \int_{\mathbb R^{2n}}e^{-p|x|^2/(4A)}V(A,x)\,dx
 =H_{A/p}[V(A,\cdot)](0).
\end{align}
If $0<A<B$, \eqref{h:comparison} yields
\[
 V(A,\cdot)\le H_{(B-A)/p}V(B,\cdot).
\]
Positivity and the semigroup identity for $H_t$ imply
\[
 \mathcal L_A^{\mathbb H_n}(f)
 \le H_{A/p}H_{(B-A)/p}[V(B,\cdot)](0)
 =\mathcal L_B^{\mathbb H_n}(f).
\]
The density of $C_c^\infty$ in
$L_p$ and \eqref{h:Lipschitz}  extend this inequality to every $f\in L_p(\mathbb H_n)$.

Next consider $A\to0^+$, again first with
$f\in C_c^\infty(\mathbb H_n)$. Rewrite \eqref{h:Lheat} as
\[
 \mathcal L_A^{\mathbb H_n}(f)
 =\pi^{-n}\int_{\mathbb R^{2n}}e^{-|x|^2}
 V\left(A,2\sqrt{\frac Ap}\,x\right)\,dx.
\]
The integrand converges pointwise to $e^{-|x|^2}V(0,0)$.
For $0<A\le1$ and $|z|\le R$ on $\supp f$, we have
\[
 V\left(A,2\sqrt{\frac Ap}\,x\right)
 \le\|f\|_p^p e^{2R\sqrt p|x|}.
\]
Dominated convergence gives
\[
 \lim_{A\to0^+}\mathcal L_A^{\mathbb H_n}(f)
 =V(0,0)=\|\mathcal{F}_{\mathbb H_n}(f)\|_q^p.
\]
The density argument passes the result to all $f\in L_p(\mathbb H_n)$.

The limit as $A\to\infty$ can be proved directly for every
$f\in L_p(\mathbb H_n)$. Put $g_A(z)=e^{-A|z|^2}$ and
\[
 d_A(\lambda)=\left(\frac{pA}{\pi}\right)^{n/p}
 |\lambda|^{n/q}\|W_\lambda(g_A)\|_{\mathcal S_q}.
\]
Proposition~\ref{lem:gaussian}, applied with $\theta=\theta_\lambda$, gives
\begin{equation}\label{h:gaussianlimit}
 \lim_{A\to\infty}d_A(\lambda)=C_p^{2n},\qquad\lambda\ne0.
\end{equation}
For almost every $s\in\mathbb C^n$, put
\[
 R_A(s)=\left(\frac{pA}{\pi}\right)^{n/p}
 \left\|\mathcal{F}_{\mathbb H_n}(f(z,t)g_A(z-s))\right\|_q,
 \quad
 Q_A(s)=\left(\frac{pA}{\pi}\right)^{n/p}
 \left\|\mathcal{F}_{\mathbb H_n}(f(s,t)g_A(z-s))\right\|_q.
\]
The reverse triangle inequality and the non-sharp Hausdorff-Young inequality give
\[
 |R_A(s)-Q_A(s)|^p
 \le\left(\frac{pA}{\pi}\right)^n
 \int_{\mathbb H_n}|f(z,t)-f(s,t)|^p e^{-pA|z-s|^2}\,dz\,dt.
\]
Integration in $s$ with $v=z-s$ gives
\begin{equation}\label{h:approximation}
 \|R_A-Q_A\|_p^p
 \le\left(\frac{pA}{\pi}\right)^n
 \int_{\mathbb C^n}e^{-pA|v|^2}
 \|f(\,\cdot+v,\cdot\,)-f\|_p^p\,dv\longrightarrow0.
\end{equation}
On the other hand, the translation identity \eqref{h:translation} gives
\[
 W_\lambda(g_A(\,\cdot-s))
 =\pi_\lambda(s/2,0)W_\lambda(g_A)\pi_\lambda(s/2,0).
\]
Consequently,
\[
Q_A(s)
=
\left(
\int_{\mathbb R}
|\widehat f_t(s,-2\pi\lambda)|^q
d_A(\lambda)^q\,d\lambda
\right)^{1/q}=
\left(
\int_{\mathbb R}
|\widehat f_t(s,\xi)|^q
d_A\left(-\frac{\xi}{2\pi}\right)^q\,d\mu(\xi)
\right)^{1/q}.
\]
By \eqref{h:gaussianlimit} and dominated convergence,
\[
Q_A(s)\longrightarrow
C_p^{2n}
\left(
\int_{\mathbb R}
|\widehat f_t(s,\xi)|^q\,d\mu(\xi)
\right)^{1/q}
\quad\text{in }L_p(\mathbb C^n,ds).
\]
Together with \eqref{h:approximation}, this gives
\[
 \lim_{A\to\infty}\mathcal L_A^{\mathbb H_n}(f)
 =\lim_{A\to\infty}\|R_A\|_p^p=C_p^{2np}M(f)^p.
\]
Here we complete the proof.
\end{proof}

\begin{theorem}\label{h:main-HY}
For $1\le p\le2$ and every $f\in L_p(\mathbb H_n)$,
\[
 \|\mathcal{F}_{\mathbb H_n}(f)\|_q\le C_p^{2n+1}\|f\|_{L_p(\mathbb H_n)}.
\]
The constant $C_p^{2n+1}$ is sharp.
\end{theorem}

\begin{proof}
For $1<p<2$, monotonicity and the two limits in
Theorem~\ref{h:monotonicity} give
\[
 \|\mathcal{F}_{\mathbb H_n}(f)\|_q\le C_p^{2n}M(f).
\]
Applying the classical sharp Hausdorff-Young inequality (Theorem~\ref{main-HY} with $d=1$ and $\theta=0$) to
$t\mapsto f(z,t)$ gives
\[
M(f)^p
=
\int_{\mathbb C^n}
\left(
\int_{\mathbb R}
|\widehat f_t(z,\xi)|^q\,d\mu(\xi)
\right)^{p/q}\,dz\le
C_p^p
\int_{\mathbb C^n}\int_{\mathbb R}
|f(z,t)|^p\,dt\,dz.
\]
This proves the inequality. The cases $p=1$ and $p=2$ trivially hold.

To prove sharpness for $1<p<2$, put
\[
 f_A(z,t)=e^{-A|z|^2}e^{-\pi t^2}.
\]
With the notation in \eqref{h:gaussianlimit}, we have
\[
\frac{\|\mathcal F_{\mathbb H^n}(f_A)\|_q}{\|f_A\|_p}
=
\frac{
\left(
\int_{\mathbb R}
e^{-q\xi^2/(4\pi)}
d_A\left(-\frac{\xi}{2\pi}\right)^q\,d\mu(\xi)
\right)^{1/q}
}{
\|e^{-\pi t^2}\|_{L_p(\mathbb R,dt)}
}
\longrightarrow C_p^{2n+1}.
\]
The limit follows from \eqref{h:gaussianlimit} and dominated
convergence. At $p=2$, sharpness follows from Plancherel identity.
At $p=1$, the same Gaussian satisfies
\[
 \|\mathcal{F}_{\mathbb H_n}(f_A)(\lambda)\|_\infty
 =e^{-\pi\lambda^2}
 \left(\frac{\pi}{A+\pi|\lambda|/2}\right)^n
 \longrightarrow\left(\frac\pi A\right)^n=\|f_A\|_1
 ,\qquad \lambda\to0.
\]
Thus the constant is sharp for all $1\le p\le2$.
\end{proof}

\section{The affine group}\label{sec-affine}

Let us turn to the proof for the sharp Hausdorff--Young inequality on the affine group.

\subsection{Harmonic analysis on the affine
group}
We identify the affine group $G$ with $\R\times\R_+$, equipped with
the group law
\[
 (b,a)(b',a')=(b+ab',aa').
\]
Its left and right Haar measures are $db\,da/a^2$ and $db\,da/a$,
respectively, and its modular function is $\Delta(b,a)=a^{-1}$.
We use right Haar measure $d\mu_R(b,a)=db\,da/a$ throughout this
section. The Fourier transform conventions below follow
\cite[Section 4]{KleinRusso1978}; see also \cite{Fuhr2006}.

The two representations of $G$ acting on
$\mathcal H=L_2(\R_+,dt/t)$ are given by
\[
 (\pi_\pm(b,a)\xi)(t)=e^{\mp2\pi ibt}\xi(at),
 \qquad \xi\in\mathcal H.
\]
Let $M_t$ be the positive operator $(M_t\xi)(t)=t\xi(t)$. For
$f\in L_1(G,\mu_R)$, put
\[
 \pi_{\pm,R}(f)=\int_G f(b,a)\pi_\pm(b,a)\,d\mu_R(b,a).
\]
For $1<p\le2$ and $q=p/(p-1)$, define initially for $f\in C_c^\infty(G)$
\begin{equation}\label{a:transform}
 \mathcal J_q(f)
 =\left(\overline{M_t^{1/q}\pi_{+,R}(f)},
         \overline{M_t^{1/q}\pi_{-,R}(f)}\right).
\end{equation}
Here the bars denote operator closures. The target von Neumann algebra
and its trace are
\[
 \mathcal M_G=\mathcal B(\mathcal H)\oplus\mathcal B(\mathcal H),
 \qquad \tau_G=\Tr\oplus\Tr.
\]
Thus
\[
 \|\mathcal J_q(f)\|_q^q
 =\|\overline{M_t^{1/q}\pi_{+,R}(f)}\|_{\mathcal S_q}^q+\|\overline{M_t^{1/q}\pi_{-,R}(f)}\|_{\mathcal S_q}^q.
\]
At $p=1$, we use $M_t^{1/\infty}=I$ and $\|\mathcal J_\infty(f)\|_\infty$ the
maximum of $\|\overline{\pi_{+,R}(f)}\|$ and $\|\overline{\pi_{-,R}(f)}\|$. This naturally gives
\[
\|\mathcal J_\infty(f)\|_\infty
=\max_\pm\|\pi_{\pm,R}(f)\|
\le\|f\|_{L^1(G,\mu_R)}.
\]

Write
\[
 F(b,u)=f(b,e^u).
\]
Then $\|f\|_{L_p(G,\mu_R)}=\|F\|_{L_p(\R^2)}$.
This change of coordinates make the scale variable and its measure compatible with ordinary Gaussian and heat-flow calculations, which will play an important role in the constitution of the flow \eqref{a:localization}. 
Write also
\[
\mathcal T_q(F)=\mathcal J_q(f).
\]
In order to compute the Schatten norms of $\mathcal T_q(F)$, we need to express the two components of $\mathcal T_q(F)$ by integral kernels (at least for $F\in C_c^\infty(\R^2)$). To this end, we introduce the unitary
\[
\mathcal U:\mathcal H=L^2(\mathbb R_+,dt/t)\longrightarrow L^2(\mathbb R,dx),
\qquad (\mathcal U\xi)(x)=\xi(e^x).
\]
Direct computation gives the kernels of \(\mathcal U\overline{M_t^{1/q}\pi_{\pm,R}(f)} \mathcal U^*\)
as
\begin{equation}\label{a:kernel}
 K_{q,\pm}(x,y)=e^{x/q}\widehat F_b(\pm 2\pi e^x,y-x), 
\end{equation}
where
\[
 \widehat F_b(\xi,u)=\int_\R F(b,u)e^{-ib\xi}\,db.
\]
All \(L_p\)-norms of \(\widehat F_b(\cdot,u)\) are taken with
respect to \(d\mu(\xi)=(2\pi)^{-1}d\xi\).

Since $\mathcal U$ is a unitary, the Schatten norms of \(\mathcal U\overline{M_t^{1/q}\pi_{\pm,R}(f)} \mathcal U^*\) and \(\overline{M_t^{1/q}\pi_{\pm,R}(f)} \) equal.
In particular, the Plancherel formula reads
\[
 \|\mathcal T_2(F)\|_2^2
 =\int_{\R^2}|K_{2,+}(x,y)|^2+|K_{2,-}(x,y)|^2\,dx\,dy
 =\|F\|_2^2.
\]
Interpolating the analytic family
$(M_t^{z/2}\pi_{\pm,R}(f))_\pm$  gives the non-sharp
Hausdorff--Young inequality
\begin{equation}\label{a:nonsharp}
 \|\mathcal T_q(F)\|_q\le\|F\|_p.
\end{equation}
Consequently, $\mathcal T_q:L_p(\R^2)\to L_q(\mathcal M_G,\tau_G)$
is defined by a unique bounded extension. 

We need to introduce two isometries. The trace-preserving automorphisms
\[
 T_r(X_+,X_-)
 =(M_t^{-ir}X_+M_t^{ir},M_t^{-ir}X_-M_t^{ir}),\qquad r\in\R,
\]
act isometrically on each $L_q(\mathcal M_G,\tau_G)$ and satisfy
\begin{equation}\label{a:phase}
 \mathcal T_q(e^{iru}F(b,u))=T_r(\mathcal T_q(F)).
\end{equation}
Indeed, $M_t^{-ir}\pi_\pm(b,a)M_t^{ir}=a^{ir}\pi_\pm(b,a)$.
Moreover, right translation in the second group coordinate gives
\begin{equation}\label{a:translation}
 \mathcal T_q(F(b,u-s))
 =\bigl((\mathcal T_q(F))_\pm\pi_\pm(0,e^s)\bigr)_\pm.
\end{equation}
Since $\pi_\pm(0,e^s)$ are unitary, this map preserves the
$L_q$-norm. Both identities extend to $F\in L_p(\R^2)$ by
\eqref{a:nonsharp}.

For $A>0$, define
\begin{equation}\label{a:localization}
 \mathcal L_A^G(F)
 =\left(\frac{pA}{\pi}\right)^{1/2}
 \int_\R
 \|\mathcal T_q(F(b,u)e^{-A|u-s|^2})\|_q^p\,ds.
\end{equation}
We are going to prove its monotonicity and endpoint limits.
Throughout the proof, we assume $1<p<2$.

Fix a nonzero complex-valued $F\in C_c^\infty(\R^2)$. For
$\alpha\ge0$ and $\beta\in\R$, set
\[
 F_{\alpha,\beta}(b,u)=F(b,u)e^{-\alpha u^2+\beta u}
\]
and put
\begin{equation}\label{a:VI}
 V(\alpha,\beta)=\|\mathcal T_q(F_{\alpha,\beta})\|_q^p,
 \qquad I(\alpha,\beta)=\int_{\R^2}|F_{\alpha,\beta}(b,u)|^p\,db\,du.
\end{equation}
The non-sharp Hausdorff--Young inequality directly gives
\[
 0<V(\alpha,\beta)\le I(\alpha,\beta).
\]
The map $(\alpha,\beta)\mapsto F_{\alpha,\beta}$ is smooth in
$L_p(\R^2)$, which gives $(\alpha,\beta)\mapsto
\mathcal T_q(F_{\alpha,\beta})$ smooth in $L_q(\mathcal M_G,\tau_G)$
by \eqref{a:nonsharp}. Hence there is no ambiguity when taking
derivatives.

\begin{proposition}\label{a:heatInequality}
With the notation above,
\begin{equation}\label{a:heatInequalityEq}
 (\partial_\beta^2+p\partial_\alpha)V\ge0.
\end{equation}
\end{proposition}
\begin{proof}
At a fixed $(\alpha,\beta)$, write
\[
 Y=\mathcal T_q(F_{\alpha,\beta}),\qquad
 X=\mathcal T_q(uF_{\alpha,\beta}),\qquad
 C=\mathcal T_q(u^2F_{\alpha,\beta}).
\]
Differentiation in the parameters gives
\[
 \partial_\beta Y=X,\qquad \partial_\beta^2Y=C,
 \qquad \partial_\alpha Y=-C.
\]
By the translation identity \eqref{a:phase}, the map
\[
 r\longmapsto\|\mathcal T_q(e^{iru}F_{\alpha,\beta})\|_q^p
 =\Omega(\mathcal T_q(e^{iru}F_{\alpha,\beta}))
\]
is constant. The rest of the proof proceeds as Proposition \ref{prop:heatInequality}, and is thus omitted.
\end{proof}

Denoting $H_t$ the ordinary heat operator on $\R$, for $0\le\alpha\le A$, and repeating the proof of Proposition \ref{prop:comparison}, we have
\begin{equation}\label{a:comparisonEq}
 V(\alpha,\beta)\le H_{(A-\alpha)/p}[V(A,\cdot)](\beta).
\end{equation}
This inequality will be used to deduce the monotonicity of the flow \eqref{a:localization}.

The next lemma is a preparation for computing the limit of the flow \eqref{a:localization} at the endpoint $A=\infty$.
Let $M_f$ denote multiplication by $f$, and let
$C_g$ denote convolution by $g$. Both operators act on the Hilbert space $L_2(\R)$.
\begin{lemma}\label{a:gaussian}
Let $1< p\leq 2$ with $q=p/(p-1)$, and $g_A(x)=e^{-Ax^2}$. For
$f\in L_q(\R)$,
\[
 \lim_{A\to\infty}
 \left(\frac{pA}{\pi}\right)^{1/(2p)}
 \|M_fC_{g_A}\|_{\mathcal S_q(L_2(\mathbb{R}))}=C_p\|f\|_q.
\]
\end{lemma}
\begin{proof}
First assume that \(f\in C_c^\infty(\mathbb R)\). For a function $h$ on $\R$, let $T_h$ denote the Fourier
multiplier with symbol $h$, with the Fourier convention
\[
 \widehat g(\xi)=\int_\R g(x)e^{-ix\xi}\,dx,\qquad d\mu(\xi)=(2\pi)^{-1}d\xi.
\]
We make the convention that when $h(\xi)$ is mentioned as the symbol of the  Fourier
multiplier, it is always equipped with the measure $d\mu(\xi)$. Then
\[
 \widehat g_A(\xi)
 =\sqrt{\frac{\pi}{A}}\,e^{-\xi^2/(4A)},
 \qquad
 \|\widehat g_A\|_{L_q(\R,\mu)}
 =\left(\frac{\pi}{A}\right)^{1/(2p)}
   q^{-1/(2q)}.
\]

Let us first calculate the limit for $f\in C_c^\infty (\mathbb{R})$, and for $q=2j$.
Put
\[
 h_A=h(A^{-1/2}\cdot),\qquad
 h(\xi)=\sqrt{\pi}\,e^{-\xi^2/4},
 \qquad
 X_A=M_fT_{h_A}(M_fT_{h_A})^*.
\]
Then $M_fC_{g_A}=A^{-1/2}M_fT_{h_A}$. Put
$\widetilde{h}=\mathcal F^{-1}(|h|^2)$, the integral kernel of $X_A$
is
\[
K_{X_A}(x,y)=f(x)\,A^{1/2}\widetilde{h}(A^{1/2}(x-y))\,\overline{f(y)}.
\]
For every integer $j\ge1$,
the kernel of $X_A^j$ is
\[
 K_{X_A^j}(x_1,x_{j+1})
 =\int_{\R^{j-1}}
 \prod_{i=1}^j K_{X_A}(x_i,x_{i+1})\,dx_2\cdots dx_j.
\]
The trace formula gives
\begin{align*}
 \Tr X_A^j
 &=\int_\R K_{X_A^j}(x_1,x_1)\,dx_1\\
 &=\int_{\R^j}
   \prod_{i=1}^j K_{X_A}(x_i,x_{i+1})\,dx_1\cdots dx_j\\
 &=A^{j/2}\int_{\R^j}
   \prod_{i=1}^j |f(x_i)|^2
   \prod_{i=1}^j \widetilde{h}\bigl(A^{1/2}(x_i-x_{i+1})\bigr)
   \,dx_1\cdots dx_j.
\end{align*}
Here we used
\[
 \prod_{i=1}^j f(x_i)\overline{f(x_{i+1})}
 =\prod_{i=1}^j |f(x_i)|^2,
 \qquad x_{j+1}=x_1.
\]

Now set the change of variables
\[
 x_1=x,\qquad
 x_i=x+A^{-1/2}z_i\quad(1\le i\le j),
 \qquad z_1=z_{j+1}=0.
\]
Then
\[
 A^{1/2}(x_i-x_{i+1})=z_i-z_{i+1},\qquad
 dx_1\cdots dx_j
 =A^{-(j-1)/2}\,dx\,dz_2\cdots dz_j.
\]
Consequently,
\[
 A^{-1/2}\Tr X_A^j
 =\int_{\R^j}
 \prod_{i=1}^j |f(x+A^{-1/2}z_i)|^2
 \prod_{i=1}^j \widetilde{h}(z_i-z_{i+1})
 \,dx\,dz_2\cdots dz_j.
\]
For fixed $x,z_2,\ldots,z_j$, continuity of $f$ gives the pointwise convergence
\[
 \lim_{A\to \infty}\prod_{i=1}^j |f(x+A^{-1/2}z_i)|^2=|f(x)|^{2j}.
\]
Moreover, since $z_1=0$ and $\widetilde{h}\ge0$,
\begin{align}\label{control}
 \prod_{i=1}^j |f(x+A^{-1/2}z_i)|^2
  \prod_{i=1}^j \widetilde{h}(z_i-z_{i+1})\le
 \|f\|_\infty^{2j-2}|f(x)|^2
 \prod_{i=1}^j \widetilde{h}(z_i-z_{i+1}).
\end{align}
Since
\[
 \int_{\R^{j-1}}
 \prod_{i=1}^j \widetilde{h}(z_i-z_{i+1})\,dz_2\cdots dz_j
 =\underbrace{\widetilde h*\cdots*\widetilde h}_{j\text{ times}}(0)
 =\int_\R |h(\xi)|^{2j}\,d \mu(\xi)
 =\|h\|_{2j}^{2j},
\]
the integral of the right-hand side of \eqref{control} is
\[
 \|f\|_\infty^{2j-2}\|f\|_2^2\|h\|_{2j}^{2j}<\infty.
\]
Dominated convergence therefore yields, for every integer $j\ge1$,
\begin{equation}\label{a:moments}
 \lim_{A\to\infty}A^{-1/2}\Tr X_A^j
 =\int_\R |f(x)|^{2j}\,dx
   \int_{\R^{j-1}}
   \prod_{i=1}^j\widetilde{h}(z_i-z_{i+1})\,dz_2\cdots dz_j=\|f\|_{2j}^{2j}\|h\|_{2j}^{2j}.
\end{equation}

Now let us consider general $2\leq q<\infty$.
The spectra of $X_A$ are contained in a fixed interval
$[0,L]$, where we may take $L=\|f\|_\infty^2\|h\|_\infty^2$.
For every $\varepsilon>0$, choose a polynomial $P$ such that
\[
 \left|t^{q/2-1}-P(t)\right|\le\varepsilon,\qquad 0\le t\le L.
\]
It follows that
\[
 A^{-1/2}\left|
 \Tr X_A^{q/2}-\Tr(X_A P(X_A))
 \right|
 \le\varepsilon A^{-1/2}\Tr X_A
 =\varepsilon\|f\|_2^2\|h\|_2^2.
\]
The same estimate applies to the scalar integrals of
$|f(x)h(\xi)|^q$ and
$|f(x)h(\xi)|^2P(|f(x)h(\xi)|^2)$. Applying
\eqref{a:moments} to the polynomial and then letting
$\varepsilon \to 0$, we obtain
\[
 \lim_{A\to \infty }A^{-1/2}\|M_fT_{h_A}\|_{\mathcal S_q}^q
 =\|f\|_q^q\|h\|_q^q.
\]
Substituting $
\|h\|_q
=\pi^{\frac1{2p}}q^{-\frac1{2q}}$ gives 
\[ \lim_{A\to\infty}
 \left(\frac{pA}{\pi}\right)^{1/(2p)}
 \|M_fC_{g_A}\|_{\mathcal S_q}=C_p\|f\|_q,\quad f\in C_c^\infty(\R).\]
The Cwikel type estimate gives (cf. \cite[Theorem~4.1]{simon1979}, or \cite{cwikel,birman})
\begin{equation}\label{a:multiplierBound}
 \|M_fT_h\|_{\mathcal S_q}\le\|f\|_q\|h\|_q,
 \qquad 2\le q\le\infty.
\end{equation}
Hence the limit can be passed to $f\in L_q(\R)$ together with density. 
\end{proof}

\subsection{Main result on the affine group}
\begin{theorem}\label{a:monotonicity}
For every $F\in L_p(\R^2)$, the quantity
$\mathcal L_A^G(F)$ in \eqref{a:localization} is finite and
nondecreasing for $A>0$. Moreover,
\begin{equation}\label{a:limits}
 \lim_{A\to0^+}\mathcal L_A^G(F)=\|\mathcal T_q(F)\|_q^p,
 \qquad
 \lim_{A\to\infty}\mathcal L_A^G(F)
 =C_p^p\int_\R\|\widehat F_b(\cdot,s)\|_q^p\,ds.
\end{equation}
\end{theorem}
\begin{proof}
Again following a similar argument in Theorem \ref{thm:monotonicity}, one obtain, uniformly in $A$,
\begin{equation}\label{a:Lipschitz}
 \left|\mathcal L_A^{G}(F)^{1/p}
       -\mathcal L_A^{G}(E)^{1/p}\right|
 \le\|F-E\|_p
\end{equation}
by the non-sharp Hausdorff--Young inequality. Hence we can prove the result densely for $F\in C_c^\infty(\R^2)$

We prove the monotonicity first. By the construction of
$F_{\alpha,\beta}$, we have
\[
 e^{-As^2}F_{A,2As}(b,u)=F(b,u)e^{-A|u-s|^2}.
\]
Substituting this identity into \eqref{a:localization} and making
the change of variables $x=2As$, we obtain
\begin{equation}\label{a:Lheat}
 \mathcal L_A^G(F)
 =\left(\frac{p}{4\pi A}\right)^{1/2}
 \int_\R e^{-px^2/(4A)}V(A,x)\,dx
 =H_{A/p}[V(A,\cdot)](0).
\end{equation}
If $0<A<B$, \eqref{a:comparisonEq} yields
$V(A,\cdot)\le H_{(B-A)/p}V(B,\cdot)$. Positivity and the
semigroup identity for $H_t$ imply
\[
 \mathcal L_A^G(F)
 \le H_{A/p}H_{(B-A)/p}[V(B,\cdot)](0)
 =\mathcal L_B^G(F).
\]
Density and $\eqref{a:Lipschitz}$ extend this inequality
to every $F\in L_p(\R^2)$.

Next consider $A\to0^+$, again first with
$F\in C_c^\infty(\R^2)$. Rewrite \eqref{a:Lheat} as
\[
 \mathcal L_A^G(F)=\pi^{-1/2}\int_\R e^{-x^2}
 V\left(A,2\sqrt{A/p}\,x\right)\,dx.
\]
The integrand converges pointwise to $e^{-x^2}V(0,0)$. If
$F(b,u)=0$ for $|u|>R$, then, for $0<A\le1$,
\[
 V\left(A,2\sqrt{A/p}\,x\right)
 \le\|F\|_p^p e^{2R\sqrt p|x|}.
\]
Dominated convergence gives
$\lim_{A\to0^+}\mathcal L_A^G(F)=\|\mathcal T_q(F)\|_q^p$.
The density argument passes the result to all $F\in L_p(\R^2)$.

For the limit as $A\to\infty$, first take
$F\in C_c^\infty(\R^2)$ and put
\[
 R_A(s)=\left(\frac{pA}{\pi}\right)^{1/(2p)}
 \|\mathcal T_q(F(b,u)e^{-A|u-s|^2})\|_q,
\]
\[
 d_A(s)=\left(\frac{pA}{\pi}\right)^{1/(2p)}
 \|\mathcal T_q(F(b,s)e^{-A|u-s|^2})\|_q.
\]
The reverse triangle inequality and \eqref{a:nonsharp} give
\begin{align}\label{a:approximation}
 \|R_A-d_A\|_p^p
 &\le\left(\frac{pA}{\pi}\right)^{1/2}
 \int_{\R^3}|F(b,u)-F(b,s)|^p
 e^{-pA|u-s|^2}\,db\,du\,ds\notag\\
 &=\left(\frac{pA}{\pi}\right)^{1/2}
 \int_\R e^{-pAv^2}
 \|F(b,s+v)-F(b,s)\|_{L_p(db\,ds)}^p\,dv
 \longrightarrow0.
\end{align}
For fixed $s\in\R$, put
\[
 m_{s,\pm}(x)=e^{x/q}\widehat F_b(\pm 2\pi e^x,s).
\]
By \eqref{a:kernel}, the two components of
$\mathcal{U} \mathcal T_q(F(b,s)g_A(u-s)) \mathcal{U}^*$ have kernels
\[
 K_{A,s,\pm}(x,y)
 =e^{x/q}\widehat F_b(\pm 2\pi e^x,s)g_A(y-x-s)
 =m_{s,\pm}(x)g_A(y-x-s).
\]
Put the unitary operator $V_s\varphi(y)=\varphi(y+s)$.
Since $g_A$ is even,
\[
 \begin{aligned}
 (M_{m_{s,\pm}}C_{g_A}V_s\varphi)(x)
 &=m_{s,\pm}(x)\int_\R g_A(x-y)\varphi(y+s)\,dy\\
 &=m_{s,\pm}(x)\int_\R g_A(y-x-s)\varphi(y)\,dy.
 \end{aligned}
\]
Consequently,
\[
 \mathcal{U} \mathcal T_q(F(b,s)g_A(u-s)) \mathcal{U}^*
 =\bigl(M_{m_{s,+}}C_{g_A}V_s,\,
        M_{m_{s,-}}C_{g_A}V_s\bigr),
\]
and hence
\[
 d_A(s)
 =\left(\frac{pA}{\pi}\right)^{1/(2p)}
 \left(
   \|M_{m_{s,+}}C_{g_A}\|_{\mathcal S_q}^q
   +\|M_{m_{s,-}}C_{g_A}\|_{\mathcal S_q}^q
 \right)^{1/q},
\]
where the $\mathcal{S}_q$ norm is taken on $L_2(\mathbb{R})$.
Lemma~\ref{a:gaussian} therefore gives
\[
 \lim_{A\to\infty}d_A(s)=C_p(\|m_{s,+}\|_q^q+\|m_{s,-}\|_q^q)^{1/q}.
\]
Moreover, the change of variables $\xi=2\pi e^x$ gives
\[
 \begin{aligned}
 \|m_{s,+}\|_q^q+\|m_{s,-}\|_q^q
 &=\int_\R e^x
   \left(
     |\widehat F_b(2\pi e^x,s)|^q
     +|\widehat F_b(-2\pi e^x,s)|^q
   \right)\,dx\\
 &=\int_0^\infty
   \left(
     |\widehat F_b(\xi,s)|^q
     +|\widehat F_b(-\xi,s)|^q
   \right)\,d \mu(\xi)\\
 &=\|\widehat F_b(\cdot,s)\|_q^q.
 \end{aligned}
\]
The function $C_p\|\widehat F_b(\cdot,s)\|_q$ belongs to $L_p(\R,ds)$.
Dominated convergence and \eqref{a:approximation} imply
\[
 R_A(s)\longrightarrow C_p\|\widehat F_b(\cdot,s)\|_q
 \quad\hbox{in }L_p(\R,ds).
\]
Since $\|R_A\|_p^p=\mathcal L_A^G(F)$, this proves the second
limit on $C_c^\infty(\R^2)$. Density argument extends the result to
every $F\in L_p(\R^2)$.
\end{proof}

\begin{theorem}\label{a:main-HY}
For every $f\in L_p(G,\mu_R)$,
\[
 \|\mathcal J_q(f)\|_{L_q(\mathcal M_G,\tau_G)}
 \le C_p^2\|f\|_{L_p(G,\mu_R)},\qquad 1\le p\le2.
\]
The constant $C_p^2$ is sharp.
\end{theorem}
\begin{proof}
Let $1<p<2$ and write $F(b,u)=f(b,e^u)$. Monotonicity and the
two limits in Theorem~\ref{a:monotonicity} give
\[
 \|\mathcal T_q(F)\|_q
 \le C_p\left(\int_\R
              \|\widehat F_b(\cdot,s)\|_q^p\,ds\right)^{1/p}.
\]
Applying the classical sharp Hausdorff-Young inequality (Theorem~\ref{main-HY} with $\theta=0$ and $d=1$) to
$F(\cdot,s)$ yields
\[
 \|\mathcal T_q(F)\|_q\le C_p^2\|F\|_p
 =C_p^2\|f\|_{L_p(G,\mu_R)}.
\]

To prove sharpness, put
\[
 F_A(b,u)=e^{-\pi b^2}e^{-Au^2},\qquad
 f_A(b,a)=e^{-\pi b^2}e^{-A(\log a)^2}.
\]
The two kernels in \eqref{a:kernel} are
$m_\pm(x)g_A(y-x)$, where
\[
 m_\pm(x)=e^{x/q}e^{-\pi e^{2x}},\qquad
 \left(\|m_+\|_q^q+\|m_-\|_q^q\right)^{1/q}=q^{-1/(2q)}.
\]
Since
\[
 \|f_A\|_{L_p(G,\mu_R)}
 =p^{-1/(2p)}\left(\frac{\pi}{pA}\right)^{1/(2p)},
\]
Lemma~\ref{a:gaussian} gives
\[
 \lim_{A\to\infty}
 \frac{\|\mathcal J_q(f_A)\|_q}{\|f_A\|_{L_p(G,\mu_R)}}
 =C_p\frac{q^{-1/(2q)}}{p^{-1/(2p)}}=C_p^2.
\]
The case $p=2$ is the Plancherel identity. At $p=1$, the
operator-norm estimate has constant $1$, and a nonnegative
approximate identity shows that this constant is sharp.
\end{proof}

\appendix

\section{Orthogonal basis of the quantum euclidean space}\label{app}
In this section, we fix
\[
 \theta=\mathfrak{s}=\begin{pmatrix}0&-1\\1&0\end{pmatrix}.
\]
The computation follows \cite{Gracia1988} and \cite{Gracia2004}. One can also refer to the oscillator realization described in \cite[Chapter 1]{Folland1989}. We include this part in the language of quantum euclidean space for completeness. We first recall that, for $f,g\in \mathcal{S}(\R^2)$,
\[
 \Ut(f)\Ut(g)=\Ut(f*_\theta g),
 \qquad
 U(f)^*=U(f^{*}),
\]
where
\[
 (f*_\theta g)(x)
 =
 \int_{\mathbb R^2}
 f(x-y)g(y)e^{\frac{i}{2}\langle x,\theta y \rangle}\,dy,
 \qquad
 f^*(x)=\overline{f(-x)}.
\]
In particular,
\(\mathcal S(\mathbb R^2)*_\theta
  \mathcal S(\mathbb R^2)\subseteq\mathcal S(\mathbb R^2)\).

Now we set
\[
 f_{00}(x)=\frac{1}{2\pi}e^{-|x|^2/4}.
\]
Then one can verify
$$ f_{00}*_\theta f_{00}=f_{00}.$$

Let \(\delta\) denote the Dirac distribution at the origin.
Also let \(\partial_j\delta\) denote its first derivative in the
\(j\)-th coordinate---for a test function $\varphi$,
\[
 \langle\partial_j\delta,\varphi\rangle
 =
 -\partial_j\varphi(0),
 \qquad j=1,2.
\]
Define \(
 \D_j=\Ut(\partial_j\delta)
\) on the Schwartz functions \(\mathcal S(\mathbb R^2)\).
More explicitly,
\[
 \D_j\psi
 =
 -\left.
   \frac{d}{dt}\Ut(te_j)\psi
  \right|_{t=0},
\]
where \(e_1,e_2\) are the standard coordinate vectors.
Thus
\begin{equation}\label{eq:mb-generators}
 \D_1=\partial_1+\frac{i}{2}x_2,
 \qquad
 \D_2=\partial_2-\frac{i}{2}x_1.
\end{equation}
These operators are skew-symmetric on the Schwartz functions,
and their closures are skew-adjoint.
In particular
\[
 \D_j^*=-\D_j,
 \qquad
 [\D_1,\D_2]=-iI.
\]
Equivalently, we can define
\begin{align*}
((\partial_j\delta)*_\theta f)(x)
 &=
 -\left.
  \frac{\partial}{\partial y_j}
  \left(
   e^{\frac{i}{2}\langle y,\theta x\rangle}f(x-y)
  \right)
 \right|_{y=0}=
 \partial_jf(x)-\frac{i}{2}(\theta x)_j f(x),
\\
(f*_\theta(\partial_j\delta))(x)
 &=
 -\left.
  \frac{\partial}{\partial y_j}
  \left(
   f(x-y)e^{\frac{i}{2}\langle x,\theta y\rangle}
  \right)
 \right|_{y=0}=
 \partial_jf(x)+\frac{i}{2}(\theta x)_j f(x).
\end{align*}
Therefore
\begin{equation}\label{eq:mb-left-right}
\begin{aligned}
 \D_1\Ut(f)
 &=\Ut\left(\partial_1f+\frac{i}{2}x_2f\right),
&
 \Ut(f)\D_1
 &=\Ut\left(\partial_1f-\frac{i}{2}x_2f\right),
\\
 \D_2\Ut(f)
 &=\Ut\left(\partial_2f-\frac{i}{2}x_1f\right),
&
 \Ut(f)\D_2
 &=\Ut\left(\partial_2f+\frac{i}{2}x_1f\right).
\end{aligned}
\end{equation}
All these products involving \(\D_j\) are initially defined on the
Schwartz functions and then extended to other cases.

Next, define the annihilation and creation operators by
\[
 a=\frac{\D_1+i\D_2}{\sqrt2},
 \qquad
 a^*=\frac{-\D_1+i\D_2}{\sqrt2}.
\]
The adjoint relation follows from \(\D_j^*=-\D_j\), and
\[
 [a,a^*]=i[\D_1,\D_2]=\mathbf{1}_{\mathcal S(\mathbb R^2)}.
\]

For convenience, put
\[
 z=\frac{x_1+ix_2}{\sqrt2},
 \qquad
 \partial_z=\frac{\partial_1-i\partial_2}{\sqrt2},
 \qquad
 \partial_{\bar z}=\frac{\partial_1+i\partial_2}{\sqrt2},
\]
where
\[
 |z|^2=\frac{|x|^2}{2},
 \qquad
 \partial_z z=\partial_{\bar z}\bar z=1,
 \qquad
 \partial_z\bar z=\partial_{\bar z}z=0.
\]
Formula \eqref{eq:mb-left-right} then becomes
\begin{equation}\label{eq:mb-ladder-actions}
\begin{aligned}
 a\Ut(f)
 &=\Ut\left(\partial_{\bar z}f+\frac z2f\right),
&
 a^*\Ut(f)
 &=\Ut\left(-\partial_zf+\frac{\bar z}{2}f\right),
\\
 \Ut(f)a
 &=\Ut\left(\partial_{\bar z}f-\frac z2f\right),
&
 \Ut(f)a^*
 &=\Ut\left(-\partial_zf-\frac{\bar z}{2}f\right).
\end{aligned}
\end{equation}
One can compute that
\[
 \partial_{\bar z}f_{00}=-\frac z2f_{00},
 \qquad
 \partial_zf_{00}=-\frac{\bar z}{2}f_{00}.
\]
It then follows
\begin{equation}\label{eq:mb-vacuum}
 a\Ut(f_{00})=0,
 \qquad
 \Ut(f_{00})a^*=0.
\end{equation}
More generally, for every polynomial \(P(z,\bar z)\),
\begin{equation}\label{eq:mb-polynomial-actions}
\begin{aligned}
 a\Ut(Pf_{00})
 &=\Ut\bigl((\partial_{\bar z}P)f_{00}\bigr),
\\
 a^*\Ut(Pf_{00})
 &=\Ut\bigl((\bar zP-\partial_zP)f_{00}\bigr),
\\
 \Ut(Pf_{00})a
 &=\Ut\bigl((\partial_{\bar z}P-zP)f_{00}\bigr),
\\
 \Ut(Pf_{00})a^*
 &=\Ut\bigl(-(\partial_zP)f_{00}\bigr).
\end{aligned}
\end{equation}

For \(m,n\in\mathbb N_0\), define
\begin{equation}\label{eq:mb-matrix-definition}
 E_{mn}
 =
 \frac{1}{\sqrt{m!\,n!}}(a^*)^m\circ\Ut(f_{00})\circ a^n.
\end{equation}
By \eqref{eq:mb-polynomial-actions}, every such product $E_{mn}$ equals \(\Ut(f_{mn})\)
for $f_{mn}$ a polynomial times \(f_{00}\), while $f_{mn}$ is still a Schwartz function.

\begin{proposition}\label{prop:mb-matrix-units}
For \(m,n,j,k\in\mathbb N_0\),
\[
 E_{mn}E_{jk}=\delta_{nj}E_{mk},
 \qquad
 E_{mn}^*=E_{nm}.
\]
\end{proposition}

\begin{proof}
Repeated use of \([a,a^*]=I\) gives
\[
 a^n(a^*)^j
 =
 \sum_{\ell=0}^{\min(n,j)}
 \binom n\ell\frac{j!}{(j-\ell)!}
 (a^*)^{j-\ell}a^{n-\ell}.
\]
Multiplying by \(\Ut(f_{00})\) on both sides and using
\eqref{eq:mb-vacuum}, every term vanishes unless
\(n=j=\ell\). Thus
\[
 \Ut(f_{00})a^n(a^*)^j\Ut(f_{00})=\delta_{nj}n!\Ut(f_{00}).
\]
Consequently,
\begin{align*}
 E_{mn}E_{jk}=
 \frac{(a^*)^m
      \Ut(f_{00})a^n(a^*)^j\Ut(f_{00})a^k}
      {\sqrt{m!\,n!\,j!\,k!}}=
 \delta_{nj}
 \frac{(a^*)^m\Ut(f_{00})a^k}{\sqrt{m!\,k!}}
 =
 \delta_{nj}E_{mk}.
\end{align*}
The adjoint identity follows directly.
\end{proof}

We next determine the functions \(f_{mn}\) explicitly.
By \eqref{eq:mb-polynomial-actions},
\[
 (a^*)^m\Ut(f_{00})=\Ut(\bar z^m f_{00}),
\]
and hence
\begin{equation}\label{eq:mb-symbol-derivative}
 f_{mn}
 =
 \frac{1}{\sqrt{m!\,n!}}
 \bigl((\partial_{\bar z}-z)^n\bar z^m\bigr)f_{00}.
\end{equation}
Since multiplication by \(z\) commutes with
\(\partial_{\bar z}\), the binomial formula yields
\begin{equation}\label{eq:mb-symbol-polynomial}
 f_{mn}
 =
 \frac{f_{00}}{\sqrt{m!\,n!}}
 \sum_{\ell=0}^{\min(m,n)}
 \binom n\ell
 \frac{m!}{(m-\ell)!}
 (-z)^{n-\ell}\bar z^{m-\ell}.
\end{equation}

For \(N,\alpha\in\mathbb N_0\), denote the generalized Laguerre
polynomial by
\[
 L_N^{(\alpha)}(x)
 =
 \sum_{k=0}^N
 \binom{N+\alpha}{N-k}\frac{(-x)^k}{k!}.
\]
If \(m\ge n\), taking \(k=n-\ell\) in
\eqref{eq:mb-symbol-polynomial} gives
\begin{align*}
 &\sum_{\ell=0}^{n}
 \binom n\ell
 \frac{m!}{(m-\ell)!}
 (-z)^{n-\ell}\bar z^{m-\ell}\\
 &\qquad=
 n!\,\bar z^{m-n}
 \sum_{k=0}^n
 \binom m{n-k}\frac{(-|z|^2)^k}{k!}\\
 &\qquad=
 n!\,\bar z^{m-n}L_n^{(m-n)}(|z|^2).
\end{align*}
If \(m<n\), taking \(k=m-\ell\) gives
\begin{align*}
 &\sum_{\ell=0}^{m}
 \binom n\ell
 \frac{m!}{(m-\ell)!}
 (-z)^{n-\ell}\bar z^{m-\ell}\\
 &\qquad=
 m!(-z)^{n-m}
 \sum_{k=0}^m
 \binom n{m-k}\frac{(-|z|^2)^k}{k!}\\
 &\qquad=
 m!(-z)^{n-m}L_m^{(n-m)}(|z|^2).
\end{align*}
We have therefore proved the explicit formula
\begin{equation}\label{eq:mb-laguerre}
 f_{mn}(x)
 =
 \begin{cases}
 \displaystyle
  \frac{1}{2\pi}\sqrt{\frac{n!}{m!}}\,
 \left(\frac{x_1-ix_2}{\sqrt2}\right)^{m-n}
 L_n^{(m-n)}\left(\frac{|x|^2}{2}\right)e^{-\frac{|x|^2}{4}},
 &m\ge n,
 \\[9pt]
 \displaystyle
 \frac{1}{2\pi}\sqrt{\frac{m!}{n!}}\,
 \left(-\frac{x_1+ix_2}{\sqrt2}\right)^{n-m}
 L_m^{(n-m)}\left(\frac{|x|^2}{2}\right)e^{-\frac{|x|^2}{4}},
 &m<n.
 \end{cases}
\end{equation}
The diagonal case is
\[
 f_{nn}(x)
 =
 \frac{1}{2\pi}
 L_n\left(\frac{|x|^2}{2}\right)e^{-|x|^2/4}.
\]
Moreover, $\Ut(f_{nn})$ are minimal projections and satisfy
$$ \sum_{n\geq 0}E_{nn}=\sum_{n\geq 0}\Ut(f_{nn})=\mathbf{1}.$$
The sum above converges in the strong operator topology.
Set \(P_N=\sum_{n=0}^N E_{nn}\).
The minimality of the diagonal projections and the matrix-unit
relations imply that
\[
P_N L^\infty(\mathbb R_\theta^2)P_N
=
\operatorname{span}\{E_{mn}:0\le m,n\le N\}.
\]
Since \(P_NXP_N\to X\) in \(L^2(\mathbb R_\theta^2)\) for every
\(X\in L^2(\mathbb R_\theta^2)\), the linear span of the \(E_{mn}\)
is dense in this space. Moreover,
\[
\tau_\theta(E_{mn}^*E_{jk})
=
\frac{1}{2\pi}\delta_{mj}\delta_{nk}.
\]
Thus \(\{E_{mn}:m,n\ge0\}\) is a complete orthogonal basis of
\(L^2(\mathbb R_\theta^2)\), satisfying
\[
E_{mn}E_{jk}=\delta_{nj}E_{mk},
\qquad
E_{mn}^*=E_{nm},
\]
with each \(E_{nn}\) a minimal projection.

\noindent{\bf AI Statement.} The prototype of the monotonic flow was proposed by the authors, further refined during the interactions with GPT-5.6-sol and finally reformulated by the authors; GPT-5.6-sol also makes crucial contribution to the calculation in Lemma~\ref{a:gaussian}. The authors also acknowledge the use of GPT-5.6-sol for \TeX{} editing and language polishing. The authors independently checked all the results and take full responsibility for the mathematical content of this paper.

\noindent{\bf Acknowledgement.} The authors are supported by the National Natural Science Foundation of China (Nos. 12371138 and W2441002). They would also like to thank Prof. Javier Parcet for his valuable suggestions.

\bibliographystyle{amsplain}
\bibliography{references}

@article{MSX2020,
  author  = {McDonald, Edward and Sukochev, Fedor and Xiong, Xiao},
  title   = {Quantum Differentiability on Noncommutative {Euclidean} Spaces},
  journal = {Comm. Math. Phys.},
  volume  = {379},
  number  = {2},
  year    = {2020},
  pages   = {491--542},
  doi     = {10.1007/s00220-019-03605-2}
}

@article{PotapovSukochev2014,
  author  = {Potapov, Denis and Sukochev, Fedor},
  title   = {{Fr{\'e}chet} differentiability of {$\mathcal{S}^{p}$} norms},
  journal = {Adv. Math.},
  volume  = {262},
  year    = {2014},
  pages   = {436--475},
}

@book{Folland1989,
  author    = {Folland, Gerald B.},
  title     = {Harmonic Analysis in Phase Space},
  series    = {Annals of Mathematics Studies},
  volume    = {122},
  publisher = {Princeton University Press},
  address   = {Princeton, NJ},
  year      = {1989},
  isbn      = {0-691-08528-5}
}

@article{KleinRusso1978,
  author  = {Klein, Abel and Russo, Bernard},
  title   = {Sharp inequalities for {Weyl} operators and {Heisenberg} groups},
  journal = {Math. Ann.},
  volume  = {235},
  year    = {1978},
  pages   = {175--194},
  doi     = {10.1007/BF01405012}
}

@article{CowlingEtAl2019,
  author  = {Cowling, Michael G. and Martini, Alessio and
             M{\"u}ller, Detlef and Parcet, Javier},
  title   = {The {Hausdorff--Young} inequality on {Lie} groups},
  journal = {Math. Ann.},
  volume  = {375},
  number  = {1--2},
  year    = {2019},
  pages   = {93--131},
  doi     = {10.1007/s00208-018-01799-9}
}

@article{Beckner1975,
 author = {William Beckner},
 journal = {Annals of Mathematics},
 number = {1},
 pages = {159--182},
 publisher = {Princeton University},
 title = {Inequalities in Fourier Analysis},
 volume = {102},
 year = {1975}
}

@article{Bennet2008,
 author = {Bennett, Jonathan and Bez, Neal and Carbery, Anthony},
 title = {Heat-flow monotonicity related to the {Hausdorff}-{Young} inequality},
 fjournal = {Bulletin of the London Mathematical Society},
 journal = {Bull. Lond. Math. Soc.},
 volume = {41},
 number = {6},
 pages = {971--979},
 year = {2009},
}

@book{Cartan1967,
  author    = {Cartan, Henri},
  title     = {Calcul diff\'erentiel},
  publisher = {Hermann},
  address   = {Paris},
  year      = {1967},
  series    = {Collection M\'ethodes},
  note      = {Cours de math\'ematiques, II, fascicule 1},
  language  = {French}
}

@article{Lieb1990,
    author = {Lieb, Elliott H.},
    title = {Integral bounds for radar ambiguity functions and Wigner distributions},
    journal = {Journal of Mathematical Physics},
    volume = {31},
    number = {3},
    pages = {594-599},
    year = {1990},
}

@article{Peller2006,
  author  = {Peller, Vladimir V.},
  title   = {Multiple operator integrals and higher operator derivatives},
  journal = {J. Funct. Anal.},
  volume  = {233},
  number  = {2},
  year    = {2006},
  pages   = {515--544},
}

@article{Gracia1988,
    author = {Gracia-Bondía, José M. and Várilly, Joseph C.},
    title = {Algebras of distributions suitable for phase-space quantum mechanics. I},
    journal = {Journal of Mathematical Physics},
    volume = {29},
    number = {4},
    pages = {869-879},
    year = {1988},
}

@book{qmath,
 author = {Takhtajan, Leon A.},
 title = {Quantum mechanics for mathematicians},
 fseries = {Graduate Studies in Mathematics},
 series = {Grad. Stud. Math.},
 volume = {95},
 year = {2008},
 publisher = {Providence, RI: American Mathematical Society (AMS)},
}

@article{ParcetMAMS,
 author = {Gonz{\'a}lez-P{\'e}rez, Adr{\'{\i}}an Manuel and Junge, Marius and Parcet, Javier},
 title = {Singular integrals in quantum {Euclidean} spaces},
 journal = {Memoirs of the American Mathematical Society},
 volume = {1334},
 year = {2021},
}

@book{qtmath,
 author = {Hall, Brian C.},
 title = {Quantum theory for mathematicians},
 fseries = {Graduate Texts in Mathematics},
 series = {Grad. Texts Math.},
 issn = {0072-5285},
 volume = {267},
 isbn = {978-1-4614-7115-8; 978-1-4614-7116-5},
 year = {2013},
 publisher = {New York, NY: Springer},
}

@article{Gracia2004,
  title={Moyal planes are spectral triples},
  author={Gayral, Victor and Gracia-Bondia, Jose M. and Iochum, Bruno and Sch{\"u}cker, Thomas and V{\'a}rilly, Joseph C.},
  journal={Communications in mathematical physics},
  volume={246},
  number={3},
  pages={569--623},
  year={2004},
}

@article{Fedor2020,
  title={Cwikel estimates revisited},
  author={Levitina, Galina and Sukochev, Fedor and Zanin, Dmitriy},
  journal={Proceedings of the London Mathematical Society},
  volume={120},
  number={2},
  pages={265--304},
  year={2020},
}

@article{beigimeta,
  title={A Meta Logarithmic-Sobolev Inequality for Phase-Covariant Gaussian Channels},
  author={Beigi, Salman and Rahimi-Keshari, Saleh},
  journal={Annales Henri Poincar{\'e}},
  volume={26},
  number={8},
  pages={2737--2778},
  year={2025},
}

@article{DOIVNA,
title = {Differentiation of operator functions in non-commutative Lp-spaces},
journal = {Journal of Functional Analysis},
volume = {212},
number = {1},
pages = {28-75},
year = {2004},
author = {Bernardus {de Pagter} and Fedor Sukochev},
}

@article{Sukochev2019,
  author  = {Potapov, Denis and Sukochev, Fedor and
             Tomskova, Anna and Zanin, Dmitriy},
  title   = {Fr{\'e}chet differentiability of the norm of
             {$L_p$}-spaces associated with arbitrary
             von Neumann algebras},
  journal = {Trans. Amer. Math. Soc.},
  volume  = {371},
  number  = {11},
  year    = {2019},
  pages   = {7493--7532},
}

@article{Janson,
title = {On Complex Hypercontractivity},
journal = {Journal of Functional Analysis},
volume = {151},
number = {1},
pages = {270-280},
year = {1997},
author = {Svante Janson},
}

@article{Paata,
title = {From discrete flow of Beckner to continuous flow of Janson in complex hypercontractivity},
journal = {Journal of Functional Analysis},
volume = {276},
number = {9},
pages = {2716-2730},
year = {2019},
author = {Paata Ivanisvili and Alexander Volberg},
}

@article{BakryLedoux1996,
  author  = {Bakry, Dominique and Ledoux, Michel},
  title   = {{L{\'e}vy--Gromov}'s isoperimetric inequality
             for an infinite dimensional diffusion generator},
  journal = {Inventiones Mathematicae},
  volume  = {123},
  number  = {2},
  pages   = {259--281},
  year    = {1996},
}

@article{BennettBez2009,
  author  = {Bennett, Jonathan and Bez, Neal},
  title   = {Closure properties of solutions to heat inequalities},
  journal = {Journal of Geometric Analysis},
  volume  = {19},
  number  = {3},
  pages   = {584--600},
  year    = {2009},
}

@article{Fuhr2006,
  author  = {F{\"u}hr, Hartmut},
  title   = {{Hausdorff--Young} inequalities for group extensions},
  journal = {Canad. Math. Bull.},
  volume  = {49},
  number  = {4},
  year    = {2006},
  pages   = {549--559},
}

@book{simon1979,
  author    = {Simon, Barry},
  title     = {Trace Ideals and Their Applications},
  series    = {London Mathematical Society Lecture Note Series},
  number    = {35},
  publisher = {Cambridge University Press},
  address   = {Cambridge},
  year      = {1979},
  pages     = {134}
}

@article {birman,
    AUTHOR = {Birman, Mikhail and Solomjak, Mikhail},
     TITLE = {Estimates for the singular numbers of integral operators},
   JOURNAL = {Uspehi Mat. Nauk},
  FJOURNAL = {Akademija Nauk SSSR i Moskovskoe Matemati\v ceskoe Ob\v s\v
              cestvo. Uspehi Matemati\v ceskih Nauk},
    VOLUME = {32},
      YEAR = {1977},
    NUMBER = {1(193)},
     PAGES = {17--84, 271},
}

@article {cwikel,
    AUTHOR = {Cwikel, Michael},
     TITLE = {Weak type estimates for singular values and the number of
              bound states of {S}chr\"odinger operators},
   JOURNAL = {Ann. of Math. (2)},
  FJOURNAL = {Annals of Mathematics. Second Series},
    VOLUME = {106},
      YEAR = {1977},
    NUMBER = {1},
     PAGES = {93--100},
}

@article {babenko,
    AUTHOR = {Babenko, Konstantin I.},
     TITLE = {An inequality in the theory of {F}ourier integrals},
   JOURNAL = {Izv. Akad. Nauk SSSR Ser. Mat.},
  FJOURNAL = {Izvestiya Akademii Nauk SSSR. Seriya Matematicheskaya},
    VOLUME = {25},
      YEAR = {1961},
     PAGES = {531--542},
}
\end{document}